\documentclass[a4paper,11pt]{amsart}

\usepackage{amsmath,amssymb,amsfonts,graphicx}
\usepackage{amsmath,amssymb,amsfonts,graphicx,color,fancyhdr}
\usepackage[latin1]{inputenc}
\usepackage{psfrag}

\newtheorem{theorem}{Theorem}[section]
\newtheorem{proposition}[theorem]{Proposition}
\newtheorem{lemma}[theorem]{Lemma}
\newtheorem{definition}[theorem]{Definition}

\newtheorem{remark}[theorem]{Remark}

\usepackage[colorlinks=true, pdfstartview=FitV, linkcolor=blue, citecolor=blue, urlcolor=blue]{hyperref}
\usepackage{color}
\usepackage{psfrag,caption}

\def\be#1 {\begin{equation} \label{#1}}
\newcommand{\ee}{\end{equation}}

\def\sqw{\hbox{\rlap{\leavevmode\raise.3ex\hbox{$\sqcap$}}$%
\sqcup$}}
\def\findem{\ifmmode\sqw\else{\ifhmode\unskip\fi\nobreak\hfil
\penalty50\hskip1em\null\nobreak\hfil\sqw
\parfillskip=0pt\finalhyphendemerits=0\endgraf}\fi}

\newcommand{\mb}{\medskip\noindent}
\newcommand{\gb}{\bigskip\noindent}
\newcommand{\R}{{\mathbb {R}}}

\newcommand{\N}{{\mathbb N}}
\newcommand{\Z}{{\mathbb Z}}

\newcommand{\s}{\mathbf S}

\newcommand{\U}{{\bf U}}

\newcommand{\OQ}{\mathbf Q}

\newcommand{\OP}{\mathbf P}

\begin{document}

\title
[Smooth bilinear square functions]
{\bf  Boundedness of smooth bilinear square functions and applications to some bilinear pseudo-differential operators}

\author{Fr\'ed\'eric Bernicot}
\address{Fr\'ed\'eric Bernicot
\\
Laboratoire Paul Painlev\'e \\ CNRS - Universit\'e Lille 1 \\
F-59655 Villeneuve d'Ascq, France}
\email{frederic.bernicot@math.univ-lille1.fr}

\author{Saurabh Shrivastava}
\address{
Saurabh Shrivastava \\
Dept. of Mathematics \\
Indian Institute of Technology - Kanpur\\
Kanpur-208016, India}

\email{saurabhk@iitk.ac.in}


\date{\today}



\begin{abstract} This paper is devoted to the proof of boundedness of bilinear smooth square functions. Moreover, we deduce boundedness
of some bilinear pseudo-differential operators associated with symbols belonging to a subclass of $BS^0_{0,0}$.
\end{abstract}

\maketitle

\tableofcontents

\section{Introduction}

Let us begin with some important results from the linear theory of singular multiplier operators on $L^p(\R)$.
The prototype of a singular multiplier operator is the Hilbert transform, which is defined as:
$$Hf(x):=\frac{1}{\pi}\;p.v.\;\int_{\R} f(x-y)\frac{dy}{y},\; f\in \mathcal{S}(\R).$$
Or equivalently
$$\widehat{Hf}(\xi):= -i \it{sgn}(\xi) \hat{f}(\xi),\; f\in \mathcal{S}(\R),$$
where $\mathcal{S}(\R)$ denotes the Schwartz space of functions on $\R$.

\mb It is well known that $H$ maps $L^p(\R)$ into $L^p(\R)$ for $1<p<\infty$ and $L^1(\R)$ into $L^{1,\infty}(\R).$ The study of such singular integral operators comes under the ``Calder\'{o}n-Zygmund theory".

\mb Let $\omega$ be an interval in $\R.$ Denote by ${\bf1}_\omega$ the characteristic function of $\omega$ and consider the following operator
$$\pi_\omega f(x):= \int_\R {\bf 1}_\omega(\xi)\hat{f}(\xi)e^{2\pi i x\xi}d\xi.$$
\mb Remark that the operator $\pi_\omega$ truncates the frequency to the interval $\omega$ and if $\omega=[a,b], a<b,$ it has the following relation with the Hilbert transform.
$$\pi_\omega=\frac{i}{2} (M_aHM_{-a}-M_bHM_{-b}),$$
where $M$ is the modulation operator given by $M_af(x):= e^{2\pi i ax}f(x).$
\mb Hence, using the boundedness of the Hilbert transform one can easily deduce that the operator $\pi_\omega$ possesses the same
boundedness properties as the operator $H$. Moreover, the operator norm does not depend on the interval $\omega$.

\mb Observe that for $p=2,$ using Plancherel theorem we can write
$$\|f\|_{L^2(\R)}=\left\| (\sum\limits_{n\in\Z}|\pi_{\omega_n}(f)|^2)^{\frac{1}{2}} \right\|_{L^2(\R)}, $$
where $\omega_n,\; n\in \Z,$ are disjoint intervals in $\R$ such that their union is whole of $\R$.
The above equation and the uniform bound for the operators $\pi_{\omega_n}$ on $L^p(\R)$ motivate the study of boundedness properties for the square function $f \rightarrow (\sum\limits_{n\in\Z}|\pi_{\omega_n}(f)|^2)^{\frac{1}{2}}$ when $p\neq 2.$
At this point we would like to remark that the Hilbert transform has one preferred point of singularity, whereas these square functions have infinitely many points of singularity. Therefore the study of these operators is more delicate. The first result in this direction is due to Littlewood and Paley (\cite{LiP, LiP1, LiP2}). They proved that

\begin{theorem}[\cite{LiP}-\cite{LiP2}] Let $\omega_n=(-2^{n+1},-2^n]\cup[2^n,2^{n+1}),\; n\in \Z$. For $1<p<\infty$, there exist
constants $c_p$ and $C_p$ such that for all $f\in \mathcal{S}(\R)$, we have
\begin{eqnarray}\label{firstlittlewood}
c_p \|f\|_{L^p(\R)} \leq \left\|(\sum\limits_{n\in\Z}|\pi_{\omega_n}(f)|^2)^{\frac{1}{2}} \right\|_{L^p(\R)} \leq C_p \|f\|_{L^p(\R)}.
\end{eqnarray}
\end{theorem}

\mb In literature such square functions are referred as non-smooth Littewood-Paley square functions. A little about the proof of inequality ~(\ref{firstlittlewood}): First one proves the boundedness of a suitable smooth Littlewood-Paley square function. Then vector valued arguments permit to get the right side inequality in ~(\ref{firstlittlewood}) and the left side inequality is deduced using duality.

Later, in the year 1967, Carleson~\cite{Carleson} considered the non-smooth Littlewood-Paley square function associated with the sequence
$\omega_n= [n,n+1],~n\in\Z,$ and proved the following.

\begin{theorem}[\cite{Carleson}] For $2\leq p<\infty,$ there exists a constant $C_p$ such that for all $f\in \mathcal{S}(\R)$, we have
\begin{eqnarray}\label{carleson}
\left\|(\sum\limits_{n\in\Z}|\pi_{[n,n+1]}(f)|^2)^{\frac{1}{2}}\right\|_{L^p(\R)} \leq C_p \|f\|_{L^p(\R)}.
\end{eqnarray}
\end{theorem}

\mb The smooth analogue of Carleson's Littlewood-Paley square function can be defined by taking a smooth function $\phi$ with $\it{supp}{\phi}\subseteq [0,1]$ and $\phi_n(\xi)=\phi(\xi-n), \; n\in\Z,$ as $f\rightarrow(\sum\limits_{n\in\Z}|(e^{2\pi in.}\hat{\phi})\ast f|^2)^{\frac{1}{2}}.$
This smooth operator satisfies the similar $L^p$ estimates, i.e., for $2\leq p<\infty$, there exists a constant $C_p$ such that for all $f\in \mathcal{S}(\R)$, we have
\begin{eqnarray}\label{carlesonsmth}
\left\|(\sum\limits_{n\in\Z}|(e^{2\pi in.}\hat{\phi})\ast f|^2)^{\frac{1}{2}}\right\|_{L^p(\R)}\leq C_p \|f\|_{L^p(\R)}.
\end{eqnarray}
Note that in both these Littlewood-Paley inequalities (inequality ~(\ref{firstlittlewood}) and (\ref{carleson})), the sequences of
intervals have specific properties. In the first case intervals are dilates of each other by a power of $2$ whereas in the second one they are integer translates of each other. The question for other sequences of intervals remained open for quite some time. Finally, in
the year 1985 Rubio de Francia~\cite{RF} provided a positive answer towards it in its greatest generality. His result is:

\begin{theorem}[\cite{RF}] \label{rubio} Let $\omega_n$ be an arbitrary sequence of disjoint intervals in $\R$. Then for $2\leq p<\infty$, there exists a constant $C_p$ such that for all $f\in \mathcal{S}(\R)$, we have
\begin{eqnarray}\label{rubio1}
\left\| (\sum\limits_{n\in\Z}|\pi_{\omega_n}(f)|^2)^{\frac{1}{2}} \right\|_{L^p(\R)} \leq C_p \|f\|_{L^p(\R)}
\end{eqnarray}
\end{theorem}

\mb The proof of Theorem~(\ref{rubio})is quite intricate. The first step towards it is the reduction to the case of well-distributed collection and then to invoke some vector-valued Calder\'on-Zygmund theory.


\gb We now describe some of the existing results concerning bilinear square functions.

The study of bilinear multiplier operators (or multilinear operators) attracted a great deal of attention after the breakthrough of Lacey and Thiele~\cite{LT1,LT4} on Calder\'{o}n's conjecture about the boundedness of the bilinear Hilbert transform. The bilinear Hilbert transform is defined for $f,g\in \mathcal{S}(\R) $ as:
$$H(f,g)(x):=p.v.\;\frac{1}{\pi}\int_{\R} f(x+y)g(x-y)\frac{dy}{y}.$$
Or equivalently
$$H(f,g)(x):= i \int_{\R}\int_{\R}\hat{f}(\xi)\hat{g}(\eta)\it{sgn}(\xi-\eta)e^{2 \pi ix(\xi+\eta)}d\xi d\eta.$$
Lacey and Thiele proved the following.

\begin{theorem}[\cite{LT1,LT4}] \label{laceythiele} For the exponents $p_1, p_2, p_3$ satisfying $1<p_1,p_2\leq \infty$ and
$\frac{1}{p_1}+\frac{1}{p_2}=\frac{1}{p_3}<\frac{3}{2}$, there is a constant $C>0$ such that
\begin{eqnarray}\label{bilinearht}
\|H(f,g)\|_{L^{p_3}(\R)} &\leq& C \|f\|_{L^{p_1}(\R)} \|g\|_{L^{p_2}(\R)}.
\end{eqnarray}
\end{theorem}

\mb We would like to remark that the bilinear Hilbert transform $H$ has the property that
$$H(M_af, M_ag)(x)= M_{2a}H(f,g)(x).$$

\mb This property is called the {\it modulation invariance} or {\it modulation symmetry}. A very precise time-frequency analysis is required to resolve this symmetry in order to obtain $L^p$ estimates for such operators. We refer the reader to the works of Gilbert and Nahmod ~\cite{GN,GN2} and of Muscalu, Tao, and Thiele ~\cite{MTT3b,MTT3} for the study of operators closely related to the bilinear Hilbert transform.

\mb Consider the operator
\be{eq:op} \pi_\omega(f,g)(x)=\int_{\R}\int_{\R}\hat{f}(\xi)\hat{g}(\eta){\bf 1}_\omega(\xi-\eta)e^{2\pi ix(\xi+\eta)}d\xi d\eta;  \; f,g\in \mathcal{S}(\R). \ee
Note that here the one dimensional interval $\omega$ gives rise to a strip $\{(\xi,\eta)\in \R^2: \xi-\eta \in \omega \}$ in $\R^2$. So the singularity is along the lines $\xi-\eta=a$ and $\xi-\eta=b$, where $a,b$ are the endpoints of the interval $\omega$. Using the same arguments as in the linear case, one can conclude the same $L^p$ estimates for the operator $\pi_\omega$ as for the bilinear Hilbert transform, with bounds independent of the interval $\omega$.

\mb We consider the bilinear Littlewood-Paley square function $(f,g)\rightarrow (\sum\limits_{n\in\Z}|\pi_{\omega_n}(f,g)|^2)^{\frac{1}{2}}$ associated with a sequence of intervals $\{\omega_n\}_{n\in\Z}$. At this point we would like to remark that unlike the linear case there is no passage available from smooth square functions to non-smooth square functions in bilinear setting. So, in the bilinear case these are two different problems. The first result in this direction is due to Lacey~\cite{lacey2}. He considered the bilinear analogue of smooth Carleson's Littlewood-Paley square function~(\ref{carlesonsmth}) and proved the following:
\medskip

\begin{theorem}[\cite{lacey2}]\label{laceysquare} Let $\chi \in C^\infty(\R)$ with $\it{supp}{\chi}\subseteq [0,1]$. Define $\chi_{[n,n+1]}(\xi)=\chi(\xi-n), \; n\in\Z$. Then, for $2\leq p_1,p_2\leq \infty$ satisfying
$\frac{1}{p_1}+\frac{1}{p_2}=\frac{1}{2}$, there is a constant $C>0$ such that
\begin{eqnarray}\label{laceysqre}
\sum\limits_{n\in\Z}\|T_{\chi_{[n,n+1]}}(f,g)\|_{L^{2}(\R)}^2 &\leq& C \|f\|_{L^{p_1}(\R)} \|g\|_{L^{p_2}(\R)}.
\end{eqnarray}
\end{theorem}

\mb The proof of this theorem specifically depends on the exponent $2$. Mohanty and Shrivastava~\cite{MS} extended Theorem~\ref{laceysquare} for other exponents $p_3\neq2$. Their result is:

\begin{theorem}[\cite{MS}] \label{ms} Let $\phi \in \mathcal{S}(\R)$. Define $\phi_n(\xi)=\phi(\xi-n), \; n\in\Z$. Then for $2< p_1,p_2\leq \infty$ and $\frac{4}{3}<p_3\leq\infty$ satisfying
$\frac{1}{p_1}+\frac{1}{p_2}=\frac{1}{p_3}$, there is a constant $C>0$ such that
\begin{eqnarray}\label{ms1}
\left\|(\sum\limits_{n\in\Z}|T_{\phi_n}(f,g)|^2)^{\frac{1}{2}} \right\|_{L^{p_3}(\R)} &\leq& C \|f\|_{L^{p_1}(\R)} \|g\|_{L^{p_2}(\R)}.
\end{eqnarray}
\end{theorem}

\mb In the proof of this theorem, the authors bounded the square function by the bilinear Hardy-Littlewood maximal operator (studied by Lacey in~\cite{lacey}).  That is why the range of exponents depends on the one of this bilinear maximal function.
By using a particular case of the result which we have proved in this paper, we can conclude Theorem~(\ref{ms}) for the remaining exponents $p_3\in(1,2]$(See~\S~\ref{remainingindex}). The bilinear analogue of smooth dyadic square function
 has been addressed by Bernicot~\cite{B} and Diestel~\cite{Diestel}. They proved that

\begin{theorem}[\cite{B,Diestel}]
Let $\psi\in \mathcal{S}(\R)$ be non-negative, have support in $\frac{1}{2}\leq |\xi|\leq 4$ and be equal to $1$ on $1\leq |\xi|\leq 2.$ Define $\psi_n(\xi):= \psi(2^{-n}\xi), \; n\in\Z.$ Then for exponents $p_1, p_2, p_3$ satisfying $1<p_1,p_2\leq \infty$ and
$\frac{1}{p_1}+\frac{1}{p_2}=\frac{1}{p_3}<\frac{3}{2}$, there is a constant $C>0$ such that for all $f,g \in \mathcal{S}(\R)$ we have

\begin{eqnarray}\label{bilittlewoodsmth}
\left\|(\sum\limits_{n\in\Z}|T_{\psi_n}(f,g)|^2)^{\frac{1}{2}}\right\|_{L^{p_3}(\R)} &\leq& C \|f\|_{L^{p_1}(\R)} \|g\|_{L^{p_2}(\R)}.
\end{eqnarray}
\end{theorem}

\mb Here the authors have linearized the square function using Radamacher functions and showed that the resulting operator falls under the setting of
Gilbert and Nahmod~\cite{GN} and then they can deduce the desired conclusion.

\mb The $L^p$ estimates for non-smooth bilinear Littlewood-Paley square functions seem to be very hard to prove. There is only one
result known in this direction, which is due to Bernicot~\cite{B}. He proved the following.

\begin{theorem}[\cite{B}] \label{bernicot}  Let $\omega_n = [a_n, b_n]$ be a sequence of
disjoint intervals in $\R$ with $b_n - a_n
= b_{n-1}- a_{n-1}$ and $a_{n+1}- b_n = a_n - b_{n-1}$ for all
$n\in\Z$. Then for exponents $2 < p_1,p_2,p'_3 < \infty $ satisfying $\frac{1}{p_1}+ \frac{1}{p_2}= \frac{1}{p_3}$,
there is a constant $C=C(p_1,p_2,p_3)$ such that for all functions
$f,g \in \mathcal{S}(\R)$ we have
\begin {eqnarray} \left\| (\sum\limits_{n\in\Z}|\pi_{\omega_n}(f,g)|^2)^{\frac{1}{2}} \right\|_{L^{p_3}(\R)} &\leq& C \|f\|_{L^{p_1}(\R)}
\|g\|_{L^{p_2}(\R)}.
\end {eqnarray}
\end{theorem}

\mb One can easily observe that the bilinear analogue of non-smooth Carleson's Littlewood-Paley theorem~(\ref{carleson}) is a particular case of Theorem~(\ref{bernicot}). The proof of this result is based on time-frequency techniques. The author has introduced a new notion of ``vector trees, vector size, and vector energy" (see \S $3$ of~\cite{B} for precise definitions) in his work and proved appropriate estimates for these quantities to obtain the desired result.

\mb The proof of the previous theorem relies on standard time-frequency analysis (used for the bilinear Hilbert transforms) mixed with $\ell^2$-valued arguments to deal with the square function. But it is based on a geometric assumption : the intervals $\omega\in \Omega$ are exactly well-distributed, they have the same length and are equi-distant.\\
We can expect that the result remains true for more general collections $\Omega$ but the analysis seems to be very hard. That is why, in this paper we aim to study the smooth square functions, which should be easier to study.

\mb In this paper we have obtained $L^p$ estimates for smooth bilinear square functions. Some of our arguments are general and can be extended to more general situations. But, we could complete the proof of our main result only under the following assumption:

\begin{equation} \label{assumption}
\inf_{\omega \in \Omega} |\omega| \simeq \sup_{\omega \in \Omega} |\omega|.
\end{equation}

\begin{theorem} \label{thm:principa} Let $\Omega:=(\omega)_{\omega\in\Omega}$ be a well-distributed
collection of intervals satisfying (\ref{assumption}). \\
Then for exponents $p_1,p_2,p_3'\in [2,\infty]$ satisfying
$$0<\frac{1}{p_3}=\frac{1}{p_1}+\frac{1}{p_2},$$
there exists a constant $C$, independent of the collection $\Omega$, such that for all $f,g \in \mathcal{S}(\R)$, we have
$$ \left\| \left(\sum_{\omega\in\Omega} \left| T_{\chi_{\omega}}(f,g)\right|^2 \right)^{1/2}\right\|_{L^{p_3}(\R)} \leq C\|f\|_{L^{p_1}(\R)} \|g\|_{L^{p_2}(\R)}.$$
\end{theorem}

\mb By the counterexample in \cite{B}, we know that $p_1, p_2\geq 2$ is necessary to get boundedness in above theorem.

\mb Let us point out that in this work, our approach allows to consider the case of a collection of intervals with equivalent lengths. Whereas in \cite{B}, the result relies on a tricky approach based on the fact that the intervals are well-distributed (with equal lengths and equi-distant).

\gb Moreover, we would like to describe another motivation to our work, concerning bilinear pseudo-differential operators. We refer the reader to \cite{pseudo, pseudo1} for works of the first author and of B\'enyi, Torres and Nahmod \cite{bipseudo, BT}. There, several classes of bilinear pseudo-differential symbols are studied (see Section \ref{sec:bipseudo} for more details). We are specially interested in the "exotic" class $BS^0_{0,0}$ which is the following one:
a function $\sigma\in C^\infty(\R^3)$ belongs to $BS^0_{0,0}$ if
$$ \left| \partial _x^a \partial _\xi ^b \partial _\eta ^c \sigma (x, \xi ,\eta ) \right|\leq C_{a,b,c},$$
for all indices $a,b$ and $c$.
This class corresponds to the bilinear version of the linear class of symbols $S^0_{0,0}$. For such a symbol, we look for boundedness of the associated operator
$$T_\sigma(f,g)(x):= \int_{\R^2} e^{ix(\xi+\eta)} \widehat{f}(\xi)  \widehat{g}(\eta) \sigma(x,\xi,\eta) d\xi d\eta.$$
In \cite{CV}, Calder\'on and Vaillancourt proved that the linear operators associated with symbols in $S^0_{0,0}$ are bounded on $L^2(\R)$. A natural question arises for the bilinear operators ?

In this framework, a first negative answer has been given by B\'enyi and Torres in \cite{BT2}: for every exponents $p,q,r\in [1,\infty)$ satisfying H\"older rule, there exist symbols $\sigma\in BS^0_{0,0}$ (which are $x$-independent) such that the bilinear operator $T_\sigma$ is not bounded from $L^p(\R) \times L^q(\R)$ to $L^r(\R)$. \\
From this point, a question arises : which extra assumption on a symbol $\sigma\in BS^0_{0,0}$ yields the boundedness of the corresponding operator ?

This work has also the motivation to give some answers to this question for the local-$L^2$ case. We refer the reader to Section \ref{sec:bipseudo} for more details concerning existing results. \\
The class $BS^0_{0,0}$ is invariant under rotation in the two frequency variables. Unfortunately, the study of the non-smooth square functions in Theorem \ref{bernicot} is based on the time-frequency analysis used for the bilinear Hilbert transforms. It is now well-known that for such operators, there are three "forbidden" lines
$$ \{(\xi,\eta),\ \xi=0\}, \quad \{(\xi,\eta),\ \eta=0\} \quad \textrm{and} \quad \{(\xi,\eta),\ \xi+\eta=0\}.$$
So, Theorem \ref{bernicot} can be extended to the case where
symbols ${\bf 1}_\omega(\xi-\eta)$ are replaced by symbols
${\bf 1}_\omega(\xi-t \eta)$ with $t\notin\{0,-1,\infty\}$.
Since we are interested in the class $BS_{0,0}^0$ (where these three degenerate lines do not play a role), we do not want to use Theorem \ref{bernicot}. Indeed our new approach for Theorem \ref{thm:principa} will allow us to consider these ``degenerate" lines
as well as the other ones. \\
These considerations are also another motivation for Theorem \ref{thm:principa}. We will define classes $W^{1,s}_\theta(BS^0_{0,0})$ of bilinear symbols (included in $BS^0_{0,0}$) (see Definition \ref{def:bs}) and obtain the following result for the associated bilinear operators.

\begin{theorem} For every $\theta\in {\mathbb S}^1$ and $s\in(1,2]$, consider a symbol $\sigma\in W^{1,s}_\theta(BS^0_{0,0})$. Then the associated operator $T_{\sigma}$ is bounded in the local $L^2$-case, i.e., $T_{\sigma}$ is bounded from $L^p(\R) \times L^q(\R)$ into $L^r(\R)$ for exponents $p,q,r' \in[2,\infty]$ satisfying
$$ 0<\frac{1}{r}= \frac{1}{p}+\frac{1}{q}.$$
Moreover, the estimates are uniform with respect to $\theta\in {\mathbb S}^1$.
\end{theorem}

\gb The current paper is organized as follows.

\section{Notation}

For all $m \in \R$ and $p\in(1,\infty)$, we set $W^{m,p}(\R)$ for the fractional Sobolev space on $\R$, defined as the set of distributions $f\in\s'(\R)$ such that $J_m(f) \in L^{p}(\R)$, where $J_m:=\left(Id-\Delta \right)^{m/2}$. \\
For a function $f\in \s(\R)$, the Fourier transform of $f$ at the point $\xi$ is defined by
$$\widehat f(\xi)=\int_{\R} f(x) e^{-ix\cdot \xi}\,dx.$$
With this definition, the inverse Fourier transform is given by
$f^{\vee} (\xi)=(2\pi)^{-1}\widehat f(-\xi)$.

For a bounded symbol $\sigma$, the bilinear operator
$$ T_\sigma(f,g)(x)= \int_{\R^2} e^{ix(\xi+\eta)} \widehat{f}(\xi) \widehat{g}(\eta) \sigma(x,\xi,\eta) \, d\xi d\eta$$
is well-defined and gives a bounded function for each pair of functions $f$, $g$ in $\s(\R)$. Moreover for
for bounded symbols $\sigma$, the operator $T_\sigma$ maps $\s(\R) \times \s(\R)$ into $\s'(\R)$ continuously. This justifies many
limiting arguments and computations that we will perform without further comment. \\
The formal transposes, $T^{*1}$ and  $T^{*2}$ of an operator $T:\s(\R) \times \s(\R) \to \s'(\R)$ are defined by
$$\langle T^{*1}(h,g), f \rangle = \langle T(f,g), h\rangle = \langle T^{*2}(f,h), g \rangle,$$
where $\langle \cdot , \cdot \rangle$ is the usual pairing between distributions and test functions.

For an interval $\omega$ of the real line, we write $\chi_\omega$ for a ``smooth scaled version'' of the characteristic function ${\bf 1}_{\omega}$. It  means that $\chi_\omega$ is a smooth function with support contained in $\omega$ and for every integer $i\in \N$, the following estimate holds
$$ \left\| D^{(i)} \chi_\omega \right\|_{L^\infty(\R)} \leq C_i |\omega|^{-i}.$$
To $\Omega:=(\omega)_{\omega\in\Omega}$ a collection of intervals, we associate the following bilinear square operator~:
$$ S_\Omega := (f,g) \rightarrow \left(\sum_{\omega\in\Omega} \left| T_{\chi_{\omega}}(f,g)\right|^2 \right)^{1/2}.$$
This can be considered as a smooth version of the more singular bilinear operator
$$(f,g) \rightarrow \left(\sum_{\omega\in\Omega} \left| \pi_{\omega}(f,g)\right|^2 \right)^{1/2}.$$

\begin{definition} \label{def:wd} A collection $\Omega:=(\omega)$ of intervals is said to be {\it well-distributed} if
\be{wd} \sum_{\omega \in \Omega} {\bf 1}_{2 \omega} \lesssim 1. \ee
\end{definition}


\part{Boundedness of bilinear square functions}

This part is devoted to the proof of Theorem \ref{thm:principa}. First we deal with the limiting case when $r=2$.

\begin{proposition} \label{prop:lacey} For $p,q\in [2,\infty]$ satisfying
$$ \frac{1}{2}=\frac{1}{p}+\frac{1}{q},$$
there exists a constant $C$ such that for every well-distributed collection $\Omega$ satisfying (\ref{assumption}), we have
$$ \left\| S_\Omega(f,g) \right\|_{L^2(\R)} \leq C\|f\|_{L^p(\R)} \|g\|_{L^q(\R)}.$$
\end{proposition}

\mb We do not write the proof and refer the reader to Theorem 1.1 in \cite{lacey2}. Indeed in \cite{lacey2}, M. Lacey dealt with the particular case where $\Omega:=([n,n+1])_{n\in\Z}$ and $\chi_{\omega}$ are obtained by translations. However, we point out that he obtained an estimate of
$$ \left\| T_{\chi_{[n,n+1]}}(f,g)\right\|_{L^2(\R)}$$
(uniformly with respect to the symbol employed for $\chi_{[n,n+1]}$), see inequality (3.6) in \cite{lacey2}. Then, he summed up all these estimates for $n\in\Z$ and obtained the desired inequalities. So in the proof of M. Lacey, we can use different symbols $\chi_{[n,n+1]}$. Proposition \ref{prop:lacey} for the case $\Omega:=([n,n+1])_{n\in\Z}$ is proved in this way. Then we extend it to a general collection of intervals by remarking that if $\chi_{J_n}$ is a smooth scaled version of ${\bf 1}_{J_n}$ with $J_{n}\subset [n,n+1]$ and $|J_n|\simeq 1$, then it is also a smooth scaled version of ${\bf 1}_{[n,n+1]}$. Hence, we obtain Proposition \ref{prop:lacey} for any such collection of intervals. Since exponents satisfy H\"older relation, using the invariance under dilations, we conclude to the whole Proposition \ref{prop:lacey}.

\gb Moreover, using invariance under dilation we can replace Assumption (\ref{assumption}) by the following one
\begin{equation} \label{assumption2}
  \frac{1}{10} \leq \inf_{\omega \in \Omega} |\omega| \leq  \sup_{\omega \in \Omega} |\omega| \leq 10,
\end{equation}
where the scale is fixed.

\gb As a consequence of this, in order to prove Theorem \ref{thm:principa} it suffices to prove the following~:

\begin{theorem} \label{thm:principal} For $p,q,r'\in (2,\infty)$ satisfying
$$\frac{1}{r}=\frac{1}{p}+\frac{1}{q},$$
there exists a constant $C$ such that for every well-distributed collection $\Omega$ satisfying (\ref{assumption2}), we have
$$ \left\| S_\Omega(f,g) \right\|_{L^r(\R)} \leq C\|f\|_{L^p(\R)} \|g\|_{L^q(\R)}.$$
\end{theorem}

\begin{remark} \label{rem:impp2} Let us point out two improvements.
\begin{itemize}
\item The proof is based on a time-frequency decomposition where the functions $f,g$ and the symbols will be decomposed with  wavelets. So we can replace the symbol $\chi_\omega (\eta-\xi)$ (defining the bilinear multiplier operator $\pi_\omega$) by any symbol $m_\omega(\eta,\xi)$ supported in $\{\eta-\xi\in \omega\}$ and such that
$$ \left\| \partial_{(\xi,\eta)}^\alpha m_\omega\right\|_{L^\infty} \leq C_\alpha$$
for all multi-index $\alpha$.

\item Moreover, we let the line $\eta-\xi=0$ play a role in the singularity of the  symbols considered. Indeed we can consider other lines. We refer the reader to Remark \ref{rem:impp} for the fact that we can consider the lines
$$\{ (\xi,\eta),\ \lambda_1\xi+\lambda_2 \eta=0\}$$
as soon as $\lambda_1\neq \lambda_2$.

\end{itemize}

\end{remark}

\mb
The next two sections are devoted to the proof of Theorem~\ref{thm:principal}.

\section{Reduction to the study of combinatorial model sums} \label{section1}

First, we ``regularize'' the collection $\Omega$ in the following sense~:

\begin{lemma} \label{lem:wdplus} To prove Theorem \ref{thm:principal}, we can assume that the collection $\Omega$ satisfies
\be{wdplus} \sum_{\omega \in \Omega} {\bf 1}_{\kappa \omega} \lesssim 1  \ee
for a large enough parameter $\kappa \geq 2$.
\end{lemma}

\begin{proof} Let $\Omega$ be a well-distributed collection of intervals and $\kappa>2$ be fixed. For each $\omega \in \Omega$, we will consider a ``finite Whitney covering'', built as follows.
For $\omega=[c(\omega)-\frac{|\omega|}{2}, c(\omega)+\frac{|\omega|}{2}]$, we define for $i=-1,..,N-1$
$$ \omega_{i}:= c(\omega)-\frac{|\omega|}{2} + |\omega| [\frac{i}{N},\frac{i+2}{N}].$$
So we have a covering $\omega \subset \bigcup_{i=-1}^{N-1} \omega_i
$. We associate a partition of the unity consisting of smooth
functions $\chi_{\omega,i}$ supported in $\omega_{i}$ such that
\begin{itemize}
 \item $\sum\limits_{i=-1}^{N-1} \chi_{\omega,i} = 1$ on $\omega$
 \item For all $i$ and for all $n$, we have
$$ \left\| D^{(n)} \chi_{\omega,i} \right\|_{L^\infty} \lesssim |\omega_i|^{-n} \lesssim |\omega|^{-n}.$$
\end{itemize}
The last implicit constant depends on $N$, but $N$ will be later fixed by $\kappa$.  \\
So we have the following decomposition~:
$$ \left|T_{\chi_\omega}(f,g)\right| \leq \sum_{i=-1}^{N-1}  \left|T_{\chi_\omega \chi_{\omega,i}} (f,g)\right|.$$
The symbol $\chi_\omega \chi_{\omega,i}$ is a smooth cutoff relative to the subinterval $\omega_{i}$. So it remains us to check that for all $i\in\{-1,...,N-1\}$, the collection $\Omega_i:=(\omega_i)_{\omega\in \Omega}$ satisfies (\ref{wdplus}) for some integer $N$. Then, we estimate the initial operator $S_\Omega$ by the sum of $N+1$ operators $S_{\Omega_i}$ associated with collections $\Omega_i$ verifying (\ref{wdplus}). Since $N$ will be finite and chosen with respect to $\kappa$, we can conclude the proof of the Lemma. \\
We fix an index $i\in\{-1,...,N-1\}$ and study the collection $\Omega_i$. We set $N\geq 4\kappa$. First, note that we have $|\omega_i|=\frac{2}{N}|\omega|$ so $$ |\kappa \omega_i|=\frac{2\kappa }{N}|\omega| \leq \frac{1}{2}|\omega|.$$
In addition,
$$c(\omega_i)=c(\omega)-\frac{|\omega|}{2} + |\omega|\frac{i+1}{N} \in \omega.$$
Consequently, $ \kappa \omega_{i} \subset 2\omega$. Since the collection $\Omega$ is well-distributed, we conclude that the collection $\Omega_i$ satisfies (\ref{wdplus}).
\end{proof}

\mb From now on, we only consider a collection $\Omega$
satisfying (\ref{assumption2}) and (\ref{wdplus}) for some large
enough $\kappa$ (which is sufficient due to the previous lemma).

\mb  Since our bilinear operators have modulation invariance
property, to prove Theorem \ref{thm:principal} we require the
``standard'' time-frequency analysis used for the study of bilinear
operators, such as the bilinear Hilbert transform. We have to
decompose both in the frequency and in the physical space with the
notions of ``tiles'' and ``tri-tiles''.

\mb For each interval $\omega\in\Omega$, we use the duality and introduce a trilinear form as follows~: for all functions $f,g,h\in\s(\R)$
\be{eqqq} \langle T_{\chi_\omega}(f,g), h \rangle = \int_{\xi_1+\xi_2+\xi_3=0} \widehat{f}(\xi_1) \widehat{g}(\xi_2) \widehat{h}(\xi_3) \chi_\omega(\xi_2-\xi_1) d\xi. \ee


\mb We recall the notion of {\it tiles} and {\it tri-tiles} (see for example \cite{MTT3b, MTT3}) and adapt the definition to our setting.

\begin{definition} A {\it tile} is a rectangle (i.e. a product of two intervals) $I\times \omega$ of area one with $1/10\leq |I|\leq 10$. A {\it tri-tile} $s$ is a
rectangle $s= I_s \times \omega_s$, which contains three tiles $s_{i}=I_{s_i} \times \omega_{s_i}$ for $i=1,2,3$ such that
\be{conddef} \forall i\in\{1,2,3\}, \qquad I_{s_i}=I_s,\ee
 \be{Qn} 0\in \omega_{s_1} + \omega_{s_2} + \omega_{s_3},  \ee
and such that there is one (and only one) interval $\omega\in \Omega$ satisfying
 \be{Qn2} \forall\, (\xi_1,\xi_2) \in \omega_{s_1} \times \omega_{s_2}, \qquad  \xi_2-\xi_1 \in \omega \ee
and $|I_s| |\omega| \simeq 1$. Let us enumerate the collection
$\Omega:=(\omega_n)_{n\geq 0}$. Then for all $n\geq 0$ and $\OQ$ a
collection of tri-tiles; we will denote by $\OQ_n$ the sub-collection
of tri-tiles verifying (\ref{Qn2}) for the interval
$\omega=\omega_n$.
\end{definition}

\mb The condition (\ref{Qn}) describes the fact that the cube
$\omega_{s_1} \times \omega_{s_2} \times \omega_{s_3}$ meets the
plane $\{(\xi_1,\xi_2,\xi_3), \xi_1+\xi_2+\xi_3=0\}$ (over which the
integral in (\ref{eqqq}) is computed). If a cube does not
intersect this plane, then the corresponding tri-tile does not play
a role and so need not be considered. Moreover, we emphasize that as
the bilinear symbol is supposed to be smooth ``at the scale $1$''
(since all the intervals $\omega \in\Omega$ have a length equivalent
to $1$ due to Assumption (\ref{assumption2})), we need to deal with only those tiles whose
space and frequency intervals have lengths equivalent to $1$.\\
We can cover the time-frequency space
$$\R \times \left( \bigcup_{\omega \in\Omega} \{(\xi,\eta), \xi-\eta \in \omega\} \right) $$
where the square function $S_{\Omega}$ plays a role by tri-tiles.

\mb We recall the concept of {\it grid} and {\it collection} of tri-tiles~:

\begin{definition} \label{def:grid}
A collection ${\mathcal I}:= \{I\}_{I\in {\mathcal I}}$ of real intervals is called a {\it grid} if for all $k\in \Z$
\be{grid} \sum_{\genfrac{}{}{0pt}{}{I\in {\mathcal I}}{2^{k-1}\leq |I|\leq 2^{k+1}}} {\bf 1}_{I} \leq C_0{\bf 1}_{\R}, \ee
where $C_0$ is a large enough numerical constant independent of $k$ and the collection ${\mathcal I}$.
\end{definition}
The classical dyadic grid is {\it a grid} in the previous sense. Moreover, a {\it grid} has a similar structure as the dyadic grid, i.e. at each scale $2^{k}$, the collection of the intervals (whose the length is equivalent to $2^k$) of the {\it grid} is a bounded covering of $\R$.

\mb We remember the notion of {\it collection of tri-tiles} (see for example Definition 2.8 in \cite{lacey})
\begin{definition} \label{def:coll}
Let $\OQ$ be a set of tri-tiles. It is called a {\it collection of tri-tiles} if
\begin{itemize}
 \item $\left\{ I_s,\ s\in \OQ\right\} \textrm{  is a grid,}$
 \item ${\mathcal J}:=\left\{ \omega_s,\ s\in \OQ \right\} \cup \left(\bigcup_{i=1}^{3} \left\{ \omega_{s_i},\ s\in \OQ \right\}\right) \textrm{  is a grid,} $
 \item for all $\omega' \in {\mathcal J}$ and $s\in\OQ$
$$ \exists i\in\{1,2,3\},\ \ \omega_{s_i} \subsetneq \omega' \in {\mathcal J} \Longrightarrow \forall\, j\in\{1,2,3\},\ \omega_{s_j} \subset \omega'. $$
\end{itemize}
\end{definition}
The last point corresponds to a ``sparseness'' assumption about the collection ${\mathcal J}$ (including all the frequency intervals of the tri-tiles), We refer to Definition 2.1 of \cite{mtt} for more details about the notion of ``sparseness''.


\mb Now we define the wave packet for a tile.

\begin{definition} \label{def:wavepacket} Let $P=I\times \omega$ be a tile. A wave packet associated with the tile $P$ is a smooth function $\Phi_P$ satisfying
\begin{itemize}
\item Fourier support of $\Phi_P$ is contained in $ \frac{9}{10} \omega$.
\item For all indices $i\in\N$ and all $M>0$, following estimate holds
$$ \left|D^{(i)}\left[e^{ic(\omega).}\Phi_{P}\right](x)\right| \lesssim |I|^{-1/2-i} \left(1+\frac{|x-c(I)|}{|I|} \right)^{-M},$$
where $c(I)$ denotes the center of the interval $I$ and  the implicit constant depends on exponents $M$.
\end{itemize}
So, $\Phi_{P}$ is a normalized function in $L^2(\R)$, concentrated
in space around $I$ and its spectrum is exactly contained in $\omega$.
\end{definition}

\begin{remark} 
As an instance of the Heisenberg uncertainty principle, the normalization $|I||\omega| \simeq 1$ for a tile $P=I \times \omega$ is chosen. This ensures that wave packets on $P$ exist. We know that we cannot choose a function $\Phi_P$
which is perfectly localized in $P$, i.e. localized in both the time and the
frequency space simultaneously (which is the case in the Walsh model, see \cite{MTT3b}). However we
just require some decay for the function $\Phi_P$ around the physical interval $I$, but it raises some technical difficulties (see \cite{MTT3}).
\end{remark}

\mb Since in Theorem \ref{thm:principal} we are only
considering exponents strictly bigger than 1, we can use duality. Hence in order to prove Theorem \ref{thm:principal} it
suffices to obtain bounds for the following trilinear form \be{eq:trilinear}
(f,g,h) \rightarrow \sum_{n} \int T_{\chi_{\omega_n}}(f,g)(x) h_n(x)
dx \ee
for $f\in L^p(\R)$, $g\in L^q(\R)$ and $h=(h_n)_n$ a sequence belonging to $L^{r'}(\R,l^2(\N))$. \\

Then the first reduction is to pass from the ``continuous'' operators $T_{\chi_{\omega_n}}$ to discrete variants involving sums of inner products with the wave packets (introduced in Definition \ref{def:wavepacket}). This step is standard and appears essential in order for the phase plane combinatorics to work correctly. This reduction could be performed by decomposing functions $f,g,h_n$ with wavelets (in terms of Fourier series on some intervals scaled with respect to the strips associated with the frequency intervals $\omega\in\Omega$). We refer the reader to \cite{BG1} and \cite{BG2} for a precise explanation of such reduction. Consequently Theorem \ref{thm:principal} can be reduced to the following one about model sum operators. Indeed, following the ideas of \cite{BG1} and \cite{BG2}, the trilinear form (\ref{eq:trilinear}) is equal to an absolutely convergent sum of disretized trilinear forms $\Lambda_{\OQ}$ (see (\ref{eq:tril})) with different collections of tri-tiles $\OQ$ and collections of wave packets. Theorem \ref{thm:principal2} below gives a uniform bound for such model trilinear forms and hence similar estimates hold for (\ref{eq:trilinear}).

\begin{theorem}  \label{thm:principal2} Let $\OQ$ be a collection of tri-tiles. For functions $f,g\in\s(\R)$ and a sequence $h:=(h_n)_n$ of smooth functions, consider the following trilinear form~:
\begin{align}
 \Lambda_\OQ(f,g,h) & := \sum_{n\geq 0} \sum_{s\in \OQ_n} |I_s|^{-1/2} \left|\langle f,\Phi_{s_1}\rangle
\langle g,\Phi_{s_2}\rangle \langle h_n,\Phi_{s_3}\rangle\right|. \label{eq:tril}
\end{align}
Then for $2<p,q,r'<\infty$ satisfying
 $$\frac{1}{r}=\frac{1}{p}+\frac{1}{q},$$
there is a constant $C$ (depending only on exponents and implicit constants appearing in Definitions \ref{def:grid} and \ref{def:wavepacket}) such that for every collection of tri-tiles $\OQ$,
$$ \left|\Lambda_\OQ(f,g,h)\right| \leq C \|f\|_{L^p(\R)}  \|g\|_{L^q(\R)} \left\| \left( \sum_{n\geq 0} |h_n|^2 \right)^{1/2} \right\|_{L^{r'}(\R)}. $$
\end{theorem}

\mb We further reduce our problem using real interpolation. Since the range for exponents is an open set, the strong type estimates (corresponding to Theorem \ref{thm:principal2}) are implied by weak type estimates using multilinear Marcinkiewicz interpolation.

\begin{definition} \label{restricteddef} For  a measurable subset $E$ of $\R$, we write~:
$$F(E) := \left\{ f\in L^\infty(\R),\ \forall x\in \R,\ |f(x)|\leq {\bf 1}_{E}(x) \right\}.$$ Let $p_1,p_2,p_3$ be positive exponents.
We say that a trilinear form $\Lambda$ is {\it of weak type} $(p_1,p_2,p_3)$ if there exists
a constant $C$ such that for all measurable sets $E_1,E_2,E_3$ of finite measure and for all functions $f \in F(E_1)$, $g \in F(E_2)$ and sequences $h:=(h_n)_n$ with $\sum_{n\in\Z} |h_n|^2 \in F(E_3)$ we have
\be{typerestreint}
\left| \Lambda(f,g,h) \right| \leq C \prod_{i=1}^{3}
|E_i|^{1/p_i}. \ee The best constant in (\ref{typerestreint}) is called
the weak type bound for the trilinear form $\Lambda$ and will be denoted by $C(\Lambda)$.
\end{definition}

\mb With the help of Marcinkiewicz real interpolation theory applied to the bilinear square functions (see the work of L. Grafakos and T. Tao \cite{GTa} or Exercise 1.4.17 of \cite{Gra} for bilinear interpolation and the one of L. Grafakos and N. Kalton \cite{GK} for extension to sub-bilinear operators), weak type estimates for exponents varying in an open range yields strong type estimates. Thus, Theorem \ref{thm:principal2} is reduced to the following one (see Remark 8 of \cite{B} for similar arguments)~:

\begin{theorem} \label{thm:principal3}
Let $2< p_1,p_2,p'_3<\infty$ satisfy
$$ \frac{1}{p_1} + \frac{1}{p_2} + \frac{1}{p_3}=1.$$
Then the trilinear form $\Lambda_\OQ$ is of weak type $(p_1,p_2,p_3)$ uniformly with respect to any finite collection of tri-tiles $\OQ$ and associated collection of wave-packets.
\end{theorem}

\section{Study of these combinatorial model sums.}

To prove Theorem \ref{thm:principal3}, we organize the collection $\OQ$ into sub-collections called {\it vectorized tri-tiles} and then we study orthogonality properties between them. We emphasize that the following estimates do not depend on the collection $\OQ$.

\gb Let us recall the notion of {\it sparseness}:

\begin{definition}[Definition 4.4 in \cite{MTT3}] A collection of intervals ${\mathcal I}:=(I)_I$ is said to be {\it sparse}, if $I$and $I'$ are any two intervals in ${\mathcal I}$ with $|I'|\leq |I|$ and $10^5 I \cap 10^5 I'\neq \phi,$ then either $|I'| < 10^9|I |$ or $I=I'$. By extension, we say that a collection of tri-tiles $\OQ$ is {\it sparse} if the collections
$$ \{I_s,\ s\in \OQ\} \qquad \textrm{and} \qquad \{\omega_{s_1}, \omega_{s_2}, \omega_{s_3},\ s\in \OQ\}
$$
are sparse.
\end{definition}

\mb We leave it to the reader to check that there exists a finite number $K$ such that every collection of tri-tiles can be split into $K$ sparse sub-collections. So without loss of generality, we can assume that the considered collection is sparse.

\mb As in \cite{B}, we will require the notion of vectorized tri-tile:

\begin{definition} \label{def:tvect} Let $s\subset \OQ$ be a tri-tile and $j\in\{1,2\}$, we define $\overrightarrow{s}^j$ to be the following collection of tri-tiles:
$$ \overrightarrow{s}^j := \left\{ s'\in \OQ, s_j=s'_j\right\}.$$
Then, we define for $\{j,l\}=\{1,2\}$
$$ \overrightarrow{s}^{j,l} := \cup_{t\in \overrightarrow{s}^j} \ \overrightarrow{t}^l.$$
\end{definition}

\mb Remark that if $t\in \overrightarrow{s}^j$ then $I_{t}=I_{s}$.

\begin{definition} \label{def:jdisjoint} Let $j\in\{1,2\}$ and $D:=(\overrightarrow{s_i}^l)_i$ be a collection of $l$-vectorized tri-tiles (with $l\neq j$).
We say that $D$ is strongly $j$-disjoint if
\begin{enumerate}
 \item for all $i\neq i'$ and for all $s\in \overrightarrow{s_i}^l, s'\in \overrightarrow{s_{i'}}^l$, we have $s_j \cap s'_j= \emptyset$
 \item and for $i\neq i'$ and $s\in \overrightarrow{s_i}^l$, $s'\in \overrightarrow{s_{i'}}^l,$ if $2\omega_{s_j} \cap 2\omega_{s'_j} \neq \emptyset$ then $I_{s_{i'}} \cap I_{s_i} = \emptyset$.
\end{enumerate}
\end{definition}

\mb We now define the quantities {\it size} and {\it energy}.

\begin{definition} \label{def:size1} Let $\OP$ be a collection of tri-tiles, $f$ a function, $j\in \{1,2\}$ and $l\neq j$. We define
$$ \overrightarrow{\textrm{size}}^l_j(f,\OP) :=\sup_{s \subset \OP} |I_s|^{-1/2}\left( \sum_{s'\in \overrightarrow{s}^l} \left|\langle f,\Phi_{s'_j}\rangle\right|^2 \right)^{1/2}.$$
\end{definition}
For a sequence of functions it is defined as:
\begin{definition}
Let $\OP$ be a collection of tri-tiles and $h:=(h_n)_{n\in\Z}$ be a
sequence of functions. We define for $\{j,l\}=\{1,2\}$
$$\overrightarrow{\textrm{size}} ^{j,l} _3(h,\OP) := \sup_{s\subset \OP} |I_s|^{-1/2} \left( \sum_{n\geq 0} \ \sum_{s'\in \overrightarrow{s}^{j,l} \cap \OP_n} \left| \langle h_{n},\Phi_{s'_3}\rangle\right|\right)^{1/2}.$$
\end{definition}


\begin{definition} \label{def:energy1} Let $\OP$ be a finite collection of tri-tiles, $j\in \{1,2\}$, $l\neq j$, and $f$ be a function. We define
$$\overrightarrow{\textrm{energy}} ^{l} _j(f,\OP):=\sup_{k\in\Z}\sup_{D\subseteq\OP} 2^k \left( \sum_{s\in D} |I_s|\right)^{1/2},$$
where we take the supremum over all the collections
$D\subseteq\OP$ of $j$-disjoint $l$-vectorised tri-tiles $\overrightarrow{s}^{l}\in D$ such that
$$ 2^{2k} |I_s| \leq \sum_{s'\in
\OQ \cap \overrightarrow{s}^{l}} \left|\langle f,\Phi_{s'_j}\rangle\right|^2 \leq
2^{2k+2} |I_s|. $$
\end{definition}

\mb We define the vectorized energy for a sequence of functions as
follows:

\begin{definition} \label{def:energysequence} Let $\OP$ be a collection of tri-tiles and
$h:=(h_n)_{n}$ be a sequence of functions. Then
$$\overrightarrow{\textrm{energy}}_3(h,\OP) := \sup_{D} \left(\sum_n   \sum_{s \in \OP_n\cap D} \left|\langle h_n,\Phi_{s_j}\rangle \right|^2 \right)^{1/2}$$
where we take the supremum over all the collections $D$ of different tri-tiles.
\end{definition}

\mb In the next section we prove estimates for all these new quantities.

\subsection{Estimates of the quantities ``size'' and ``energy''.}

Let us first obtain required estimates for the quantities associated with single function.

\begin{theorem} \label{thm:size1}
Let $\OP$ be a finite collection of tri-tiles, $f$ be a function and $j\in\{1,2\}$ and $l\neq j$, then
\be{size10bb} \overrightarrow{\textrm{size}}^l_j(f,\OP) \lesssim \|f\|_{L^\infty(\R)} \ee
\end{theorem}

\begin{remark} \label{rem:impp3}
The proof given below, can be slightly simplified. Indeed,
keeping the Assumption (\ref{assumption}) in mind, we decompose the frequency plane into cubes of length equal to one. So the geometry of such a decomposition is simple. \\
The proof, we give here is general and shows how the arguments are related to the geometry of the strip. With our approach we are able to use the linear square function associated with the collection $\Omega$. This proof also holds for an arbitrary well-distributed collection of intervals.
\end{remark}

\begin{proof}
As the collection $\OP$ is finite, we may fix a tri-tile $s$ such that
\begin{eqnarray}\label{size2}
\overrightarrow{\textrm{size}}^l_j(f,\OP) & = & \left( \frac{1}{|I_s|}\sum_{s'\in\overrightarrow{s}^{l}}
\left|\langle f,\Phi_{s'_j}\rangle\right|^2 \right)^{1/2}.
\end{eqnarray}
We know that $\overrightarrow{s}^{l}$ is a collection of tri-tiles having the following properties : for each integer $n$, $\overrightarrow{s}^{l} \cap \OP_n$ is a collection of tri-tiles $s'$ whose the frequency part $\omega_{s'_1}\times \omega_{s'_2} \times \omega_{s'_3}$ satisfies
\begin{itemize}
 \item $|\omega_{s'_1}|=|\omega_{s'_2}|=|\omega_{s'_3}| \simeq |\omega_{n}|$
\item for all $(\xi_1,\xi_2) \in \omega_{s'_1} \times \omega_{s'_2}$, $\xi_2-\xi_1 \in \omega_{n}$.
\end{itemize}
 Consequently, due to the property of the {\it collection of tri-tiles} (see Definition \ref{def:coll}) and sparseness, we get that
$$ \sharp (\overrightarrow{s}^{l} \cap \OP_n) \in\{0,1\}.$$
Let denote $\omega_{s,n}$ the (eventual) frequency cube corresponding to this tri-tile. So we have the following description of $\overrightarrow{s}^{l}$
$$ \overrightarrow{s}^{l} = \bigcup_{n} I_{s} \times \omega_{s,n}.$$
Consequently, we get
\begin{eqnarray*}
\overrightarrow{\textrm{size}}^l_j(f,\OP) & \leq & \left( \frac{1}{|I_s|} \sum_{n\geq 0} \left|\langle f,\Phi_{I_s \times (\omega_{s,n})_j }\rangle\right|^2 \right)^{1/2}.
\end{eqnarray*}
Since the functions $\Phi_{I_s \times (\omega_{s,n})_j }$ are normalized in $L^2(\R)$, have Fourier support contained in $(\omega_{s,n})_j$ and have very fast decay around $I_s$, we deduce that
$$ \left|\langle f,\Phi_{ I_s \times (\omega_{s,n})_j }\rangle\right| \leq \sum_{m\geq 0} m^{-M} \left\| \pi_{(\omega_{s,n})_j}(f) \right\|_{L^2(mI_s)},$$
where integer $M$ can be taken as large as we want. We recall that $\pi_{(\omega_{s,n})_j}$ is the linear Fourier multiplier associated with the symbol ${\bf 1}_{(\omega_{s,n})_j}$. \\
Moreover, it is easy to see that there exists a point $\xi_0$ such that
\begin{equation}
 (\omega_{s,n})_{j} \subset \xi_0 \pm c'' \omega _{n} \label{eq:important}
\end{equation}
for some numerical constant $c''$. \\
Indeed, let us fix $j=1$ and $l=2$ for example. By definition, we know that
\begin{itemize}
 \item $(\omega_{s})_2 =(\omega_{s_n})_2$ where
$s_n$ is the tri-tile belonging to $\overrightarrow{s}^{2} \cap \OP_n$
 \item and if $n_0$ is the integer such that $s$ belongs to $\OP_{n_0}$, we have
$$ |(\omega_{s_n})_1| \simeq |\omega_n| \simeq |\omega_{n_0}| \simeq |(\omega_{s})_1|,$$
\end{itemize}
where we have used Assumption (\ref{assumption}). \\
Let $\eta_0$ be any point of $(\omega_s)_{2}$, then for all $n\geq 0$, we have
\begin{align*}
 (\omega_{s_n})_{1} & \subset (\omega_{s_n})_{1}-(\omega_{s_n})_{2}+(\omega_{s_n})_{2} \\
                     & \subset - \omega_{n} + (\omega_{s_n})_{2} \\
                     & \subset - \omega_{n} + \eta_0-c(\omega_n)+ c\omega_n \\
                     & \subset \eta_0 -c' \omega_{n},
\end{align*}
with numerical constants $c,c'$.
The other situations (the case of integers $j\neq l$) can be studied exactly the same way. So (\ref{eq:important}) is proved. \\
Hence we conclude that
\begin{align*}
\overrightarrow{\textrm{size}}^l_j(f,\OP) & \leq \sum_{m\geq 0} m^{-M} \frac{1}{|mI_s|^{1/2}} \left\| \left(\sum_{n\geq 0} |\pi_{\xi_0 \pm c''\omega_{n}} (f)|^2 \right)^{1/2} \right\|_{L^2(mI_s)}.
\end{align*}
The point $\xi_0$ acts as modulation and does not depend on $n$, so we get
$$ \overrightarrow{\textrm{size}}^l_j(f,\OP)  \leq \sum_{m\geq 0} m^{-M}  \frac{1}{|m I_s|^{1/2}} \left\| \left(\sum_{n\geq 0} |\pi_{\pm c'' \omega_{n}} (e^{i\xi_0 .} f)|^2 \right)^{1/2} \right\|_{L^2(mI_s)}.$$
Thus the linear square function relative to the collection $(\pm c'' \omega_n)_{n\geq 0}$ comes into the picture as remarked earlier. Since the collection $(\omega_{n})_{n\geq 0}:=\Omega$ is assumed to satisfy (\ref{wdplus}) for some large enough $\kappa$ (see Lemma \ref{lem:wdplus}), the collection $(\pm c'' \omega_n)_{n\geq 0}$ is still a bounded covering of $\R$. So we can apply Rubio de Francia's inequality (see \cite{RF}) to the linear square function
\be{eq:squuare} f \rightarrow \left(\sum_{n\geq 0} |\pi_{\pm c'' \omega_{n}}(f)|^2 \right)^{1/2}. \ee
Indeed we will just apply the easy part for $p=2$. At this point we would like to point out that the appearance of this linear square function is the main reason for considering only the ``local-$L^2$ range'' for the exponents in the Theorem \ref{thm:principal}. \\
Let us now decompose $f$ as follows:
$$ f = f {\bf 1}_{4I_s} + \sum_{b\geq 2} f {\bf 1}_{2^{b+1}I_s \setminus 2^b I_s}.$$
For the first term, with a large enough integer $M$, we get that
$$
 \sum_{m\geq 0} m^{-M+4} \frac{1}{|mI_s|^{1/2}} \left\| \left(\sum_{n\in\Z} |\pi_{\pm c'' \omega_{n}} (e^{i\xi_0.} f{\bf 1}_{4I_s} )|^2 \right)^{1/2} \right\|_{L^2(m I_s)} \lesssim \frac{1}{|I_s|^{1/p}} \|f\|_{L^p(4I_s)} \lesssim \|f\|_{L^\infty}.$$
For the terms $f {\bf 1}_{2^{b+1}I_s \setminus 2^b I_s}$ with $b\geq 2$, we use the same reasoning by noticing the following fact
$$\langle f {\bf 1}_{2^{b+1}I_s \setminus 2^b I_s},\Phi_{s_j} \rangle = 2^{-2bM_1} \left\langle \frac{2^{2bM_1}}{(1+\frac{d(.,I_s)}{|I_s|})^{2M_1}} f {\bf 1}_{2^{b+1}I_s \setminus 2^b I_s}, (1+\frac{d(.,I_s)}{|I_s|})^{2M_1} \Phi_{s_j} \right\rangle.$$
We remark that the function $(1+\frac{d(.,I_s)}{|I_s|})^{2M_1} \Phi_{s_j}$ is always a wave packet for the tile $s_j$ (according to Definition \ref{def:wavepacket}). As $(1+\frac{d(x,I_s)}{|I_s|})\simeq 2^{b}$ for $x$ in $2^{b+1}I_s \setminus 2^b I_s$, the same arguments hold
as previously and we get an extra factor $2^{-2bM_1}$ with a power $M_1$ as large as we want. Consequently we obtain that
\begin{align*}
\overrightarrow{\textrm{size}}^l_j(f,\OP) & \leq \|f\|_{L^\infty(\R)} + \sum_{b\geq 2} 2^{-2bM_1} \|f\|_{L^\infty(\R)} \lesssim \|f\|_{L^\infty(\R)},
\end{align*}
which concludes the proof of (\ref{size10bb}).
\end{proof}

\begin{theorem} \label{thm:energy1} Let $\OP$ be a finite collection of tri-tiles. For $j\in\{1,2\}$, $l\neq j$ and a function $f$, we have (independent of the collection $\OP$)
$$ \overrightarrow{\textrm{energy}} ^{l} _j(f,\OP) \lesssim \|f\|_{L^2(\R)}.$$
\end{theorem}

\mb Here we follow the arguments given in \cite{MTT3} and provide the detail of the proof for easy reference. We first prove some lemmas.

\begin{lemma}\label{energylemma1} Let $\OP$ be a collection of tri-tiles and set $a_{s_j}:=\langle f ,\Phi_{s_j} \rangle$. Then there exists a collection $\U$ of strongly $j$-disjoint vectorized tri-tiles $(\overrightarrow{s_i} ^{l})_i$ and
complex numbers $c_{s_j}$ for all $s\in \cup_{i} \overrightarrow{s_i} ^{l}$ such
that
$$\overrightarrow{\textrm{energy}} ^{l} _j(f,\OP) \approx \left| \sum_{i} \sum_{s\in \overrightarrow{s_i} ^{l}}
a_{s_j}\bar{c}_{s_j}\right|.$$
and such that for all $i$
$$\sum_{s\in \overrightarrow{s_i} ^{l}} |c_{s_j}|^2 \lesssim \frac{|I_{s_i}|}{\sum_{i} |I_{s_i}|}.$$
\end{lemma}

\begin{proof} Since the collection $\OP$ is finite, take $k$ and $U$ be an optimizer of the vectorized quantity ``energy''. For all $s\in \cup_{i} \overrightarrow{s_i} ^{l},$ define
$$c_{s_j}:= 2^{-k}(\sum_{i }|I_{s_i}|)^{-\frac{1}{2}} a_{s_j}.$$
Then,
\begin{eqnarray*}
\sum_{i} \sum_{s\in \overrightarrow{s_i} ^{l} } a_{s_j}\bar{c}_{s_j}&=& (\sum_{i} \sum_{s\in \overrightarrow{s_i} ^{l}} |a_{s_j}|^2)2^{-k}(\sum_{i} |I_{s_i}|)^{-\frac{1}{2}} \\
&\lesssim & 2^{k}(\sum_{i} |I_{s_i}|)^{\frac{1}{2}}\\
&=& \overrightarrow{\textrm{energy}} ^{l} _j(f,\OP).
\end{eqnarray*}
The other side of inequality is proved using the second condition. Also the inequality
$$\sum_{s\in \overrightarrow{s_i} ^{l}} |c_{s_j}|^2 \lesssim \frac{|I_{s_i}|}{\sum_{i }|I_{s_i}|}$$
is immediate from the definition of $c_{s_j}$.
\end{proof}

\begin{lemma} \label{egergylemma2} Let $\U$ be a collection of strongly disjoint vectorized tri-tiles $(\overrightarrow{s_i} ^{l})_i$ in $\OP$ and for each $s\in \overrightarrow{s_i} ^{l} \in \U$, let $c_{s_j}$ be a complex number such
that
\begin{eqnarray}\label{energy2}\sum_{s\in \overrightarrow{s_i} ^{l}} |c_{s_j}|^2 \lesssim A
|I_{s_i}|,
\end{eqnarray}
for some $A>0$. Then we have
\begin{eqnarray}\label{energyl2} \left \| \sum_{i} \sum_{s\in \overrightarrow{s_i} ^{l}}
c_{s_j}\Phi_{s_j}\right \|_{L^2(\R)} &\lesssim&  (A\sum_{i} |I_{s_i}|)^{\frac{1}{2}}.
\end{eqnarray}
\end{lemma}

\begin{proof} Taking the square on both sides of~(\ref{energyl2}) it is enough
to prove that
$$\sum_{i,i'} \sum_{s\in \overrightarrow{s_i} ^{l},s'\in \overrightarrow{s_{i'}} ^{l}}
|c_{s_j}c_{s'_j}| |\langle \Phi_{s_j} ,\Phi_{s'_j} \rangle| \lesssim
A \sum_{i} |I_{s_i}|.$$
Since the Fourier transforms of $\Phi_{s_j}$ and $\Phi_{s'_j}$
are supported in $\omega_{s_j}$ and $\omega_{s'_j}$ respectively,
the term $\langle \Phi_{s_j} ,\Phi_{s'_j} \rangle$ vanishes if
$\omega_{s_j}\cap\omega_{s'_j}= \emptyset $. So we need to consider only those terms for which $\omega_{s_j}\cap\omega_{s'_j}\neq \emptyset$. As the sum in $s_i$ and $s_{i'}$
is symmetric, we can assume that $|I_{s_{i'}}|\leq |I_{s_i}|$, which implies
$|\omega_{s_j}|\leq |\omega_{s'_j}|$. \\
Using the decay property of $\Phi_{s_j}$, we have
\begin{align*}\label{phi1}
|\langle \Phi_{s_j} ,\Phi_{s'_j} \rangle| & \lesssim
\frac{|I_{s'}|^{\frac{1}{2}}}{|I_{s}|^{\frac{1}{2}}}\left(1+\frac{\textrm{dist}(I_s,I_{s'})}{|I_s|}\right)^{-100} \\
& \lesssim
\frac{|I_{s_{i'}}|^{\frac{1}{2}}}{|I_{s_i}|^{\frac{1}{2}}}\left(1+\frac{\textrm{dist}(I_{s_i},I_{s_{i'}})}{|I_{s_i}|}\right)^{-100} \\
& \lesssim \left(1+\frac{\textrm{dist}(I_{s_i},I_{s_{i'}})}{|I_{s_i}|}\right)^{-100},
\end{align*}
since $|I_{s_{i'}}|\simeq |I_{s_i}| \simeq 1$. \\
Substituting this we only need to prove that
\begin{eqnarray} \label{energy3} \sum_{i, i'} \left(1+\frac{\textrm{dist}(I_{s_i},I_{s_{i'}})}{|I_{s_i}|}\right)^{-100} \sum_{\genfrac{}{}{0pt}{}{s\in \overrightarrow{s_i} ^{l}, s'\in \overrightarrow{s_{i'}} ^{l}}{
\genfrac{}{}{0pt}{}{\omega_{s_j}\cap\omega_{s'_j}\neq \emptyset}{|\omega_{s_j}|\leq
|\omega_{s'_j}|}}}  |c_{s_j}c_{s'_j}|
\lesssim A \sum_{i} |I_{s_i} |.
\end{eqnarray}
Since all the tiles have equivalent size, we know that $|\omega_{s_j}|\thickapprox |\omega_{s'_j}|$. Then we
estimate using the following
$$|c_{s_j}c_{s'_j}| \lesssim |c_{s_j}|^2+|c_{s'_j}|^2.$$
It is enough to estimate the
contribution of the first term $|c_{s_j}|^2$, as the second one is
similar. For a fixed $\omega_{s_j}$, the above condition and the sparseness imply that necessarily $\omega_{s'_j}=\omega_{s_j}$. Hence by the definition of $j$-disjoint vectorized tri-tiles (because  $s'\in \overrightarrow{s_{i'}} ^{l}\neq \overrightarrow{s_i} ^{l}$, $s',s \in \overrightarrow{s_i} ^{l}$ would imply $s=s'$ due to the sparseness and the fact that $j\neq l$ and the rank $2$ of the collection of tri-tiles),
the summation over $i'$ and $s'$ have disjoint spatial intervals. As
a result, this sum just contributes a numerical constant and hence we
can estimate the right hand side of (\ref{energy3}) by
$$\sum_{i} \sum_{s\in \overrightarrow{s_i} ^{l}} |c_{s_j}|^2$$
which satisfies the desired inequality because of (\ref{energy2}).
\end{proof}

\noindent{\bf Proof of Theorem \ref{thm:energy1}:}
The proof follows applying Lemma \ref{energylemma1}, as we can find a
collection of strongly $j$-disjoint vectorized tri-tiles $\U:=\{\overrightarrow{s_i} ^{l}\}_i$, and complex coefficients $c_{s_j}$ for all $s\in
\cup_{i} \overrightarrow{s_i} ^{l}$ such that
$$\textrm{energy}_j(f,\OP)\thickapprox \left|\langle f,\sum_{i} \sum_{s\in \overrightarrow{s_i} ^{l}}
c_{s_j}\Phi_{s_j}\rangle \right| $$
and
$$\sum_{s\in \overrightarrow{s_i} ^{l}} |c_{s_j}|^2 \lesssim \frac{|I_{s_i}|}{\sum_{i} |I_{s_i}|}$$
for all indices $i$. Applying Cauchy-Schwartz inequality and Lemma \ref{egergylemma2}, we get the desired result.
\findem

\mb Now we prove similar results for the ``$l^2$-valued quantities''.

\begin{theorem} \label{thm:size2} Let $\OP$ be a finite collection of tri-tiles, $h:=(h_n)$ be a sequence of functions and $\{j,l\}=\{1,2\}$, then
\be{size10b} \overrightarrow{\textrm{size}}^{j,l}_3(h,\OP) \lesssim \left\| \|h_n\|_{l^2(\N)} \right\|_{L^\infty(\R)},\ee
where the implicit constants do not depend on the functions and the collection of tri-tiles.
\end{theorem}

\begin{proof}
 Let us choose a tri-tile $s_0$  such that
$$\overrightarrow{\textrm{size}}^{j,l}_3(h,\OP) := |I_{s_0}|^{-1/2} \left(
\sum_{n\geq 0} \sum_{s'\in \overrightarrow{s_0}^{j,l} \cap \OP_n} \left| \langle
h_{n},\Phi_{s'_3}\rangle\right|^2\right)^{1/2}.$$
Fix $n\geq 0$ and to make computation easy, assume that $j=2$ and $l=1$. We analyze the collection $U:=(\omega_{s,3})_{s\in \overrightarrow{s_0}^{2,1}\cap \OP_n}$ and prove that it consists of disjoint intervals. Consider two different intervals $\omega_{s,3}$ and $\omega_{s',3}$ in this collection. Then by definition, there exist $\tau,\tau'\in \overrightarrow{s_0}^{2}$ such that $\omega_{s,1}=\omega_{\tau,1}$ and $\omega_{s',1}=\omega_{\tau',1}$. Since the collection $\OP_n$ is of rank one, we know that $\tau$ and $\tau'$ should necessarily belong to different strips : indeed if $\tau,\tau'$ belong to the same strip, then due to the sparseness and the rank one of the collection $\overrightarrow{s}^{2}$, we deduce that $\tau=\tau'$ and so $\omega_{s,1}=\omega_{s',1}$, which with $(\omega_{s,2}-\omega_{s,1})  \cap (\omega_{s',2}-\omega_{s',1}) \neq \emptyset$ (since $s,s'$ belong to the same strip) and sparseness yields that $s=s'$. Hence $\omega_{s,3}=\omega_{s',3}$, but it is not possible. However these two strips have equivalent widths (since Assumption (\ref{assumption2})). Due to the well-distributed property (\ref{wdplus}) and $\omega_{\tau,2}=\omega_{\tau',2}$ (by definition of the $2$-vectorized tri-tile), we know that
$$ d(\omega_{\tau,1},\omega_{\tau',1}) \geq d(\omega_p,\omega_{p'}) \geq \kappa >>1 $$
where $p$ denotes the strip of $\tau$ and $p'$ the strip of $\tau'$ and $\kappa>>1$ as introduced in Lemma \ref{lem:wdplus}. Finally, we deduce that
\begin{equation} d(\omega_{s,1},\omega_{s',1}) \geq \kappa >>1 \label{eq}
\end{equation}
Moreover, we know that
$$ 0\in  \omega_{s,1}+ \omega_{s,2}+\omega_{s,3} \quad \textrm{and} \quad \omega_{s,2}-\omega_{s,1} \in \omega_p $$
hence
$$ -\omega_{s,3} \subset \omega_{s,1}+ \omega_{s,1} + \omega_p. $$
Similarly,
$$ -\omega_{s',3} \subset \omega_{s',1}+ \omega_{s',1} + c\omega_{p'}. $$
Thanks to (\ref{eq}) and $|\omega_p| \simeq |\omega_{p'}|\simeq 1<< \kappa$, the two previous embeddings yield
$$ d(\omega_{s,3},\omega_{s',3}) \gtrsim 1 \simeq |\omega_{s,3}| \simeq  |\omega_{s,3}|.$$
Consequently, we obtain that the collection $U:=(\omega_{s,3})_{ s\in \overrightarrow{s_0}^{j,l}\cap \OP_n}$ is composed of disjoint intervals of equivalent length. So the corresponding linear square function
$$  h_n \rightarrow \left(\sum_{s\in \overrightarrow{s_0}^{j,l} \cap \OP_n} \left| \langle
h_{n},\Phi_{s_3}\rangle\right|^2\right)^{1/2}$$
is bounded in $L^2(\R)$, uniformly in $n$. Then we repeat the proof of Theorem \ref{thm:size1}, concerning the estimate of the ``size'' quantity,
using $L^2$-boundedness of this new square function instead of the one given by (\ref{eq:squuare}) as before. The same arguments still hold and permit us to achieve the desired result.
\end{proof}

\begin{remark} \label{rem:impp} Our aim in this remark is to describe how the arguments depend on the singular line (see Remark \ref{rem:impp2}).
If we replace the symbol $\chi_\omega (\eta-\xi)$ (defining the bilinear multiplier operator $\pi_\omega$) by any symbol $m_\omega(\eta,\xi)$ supported in $\{\lambda_1\xi-\lambda_2 \eta \in \omega\}$ and
$$ \left\| \partial_{(\xi,\eta)}^\alpha m_\omega\right\|_{L^\infty} \leq C_\alpha$$
for all multi-index $\alpha$.
Then, we leave it to the reader to check that all the previous results still hold (uniformly in $\lambda_1,\lambda_2$): indeed they are based on the decomposition of the frequency plane into cubes of length equivalent to $1$ and the proofs do not use the geometry of the strips (see Remark \ref{rem:impp3}). \\
Only the previous proof uses the structure of the strips and the arguments used degenerate when $\lambda_1=\lambda_2$. That is why Theorem \ref{thm:principal} can be extended to such symbols as soon as $\lambda_1\neq \lambda_2$.
\end{remark}

\mb We now obtain the following bounds for the quantity ``energy''~:

\begin{theorem} \label{thm:energy2} Let $\OP$ be a collection of tri-tiles and  $h=(h_n)_{n\geq 0}$ be a sequence of functions, then we have
\be{energy1bis} \overrightarrow{\textrm{energy}}_3(h,\OP) \lesssim \left\|\|h_n\|_{l^2(\N)}\right\|_{L^2(\R)}. \ee
\end{theorem}

\begin{proof} Let us fix a collection of tri-tiles $D$. Since each collection $D \cap \OP_n$ is of rank one, by sparseness the collections $D \cap \OP_n$ can be assumed to be the collections of disjoint tri-tiles of equivalent size. Then it can be easily checked that
$$ \sum_{s \in \OP_n\cap D} \left|\langle h_n,\Phi_{s_j}\rangle \right|^2 \lesssim \|h_n\|_{L^2(\R)}$$
which permits us to conclude the proof.
\end{proof}

\subsection{Proof of Theorem \ref{thm:principal3}.}

 We follow the ``standard'' reasoning employed for this kind of time-frequency analysis. So we first obtain the estimate for one vectorized tri-tile and then use the combinatorial algorithm to obtain the final estimate for the entire collection of tri-tiles.

\begin{proposition}[Tri-tile estimate] \label{prop:tree} Let $\OP$ be a collection of tri-tiles. Then for each tri-tile $s_0\in\OP$, index $\{j,l\}:=\{1,2\}$, there exists an implicit constant (independent of the tri-tile $s_0$ and the collection $\OP$) such that for all functions $f_1,f_2\in\s(\R)$ and all sequences $f_3=(f_{3,n})_n,$ we have
\begin{align*}
 \Lambda_{\overrightarrow{s_0}^{j,l}} (f_1,f_2,f_3) & :=\sum_{n\geq 0} \sum_{s\in \overrightarrow{s_0}^{j,l} \cap \OQ_n} |I_s|^{-1/2} \left|\langle f_1,\Phi_{s_1}\rangle
\langle f_2,\Phi_{s_2}\rangle \langle f_{3,n},\Phi_{s_3}\rangle\right| \\
 & \lesssim |I_s| \overrightarrow{\textrm{size}}^{2}_1(f_i,\OP) \overrightarrow{\textrm{size}}^{1}_2(f_i,\OP) \overrightarrow{\textrm{size}} ^{j,l}_3(f_3,\OP).
\end{align*}
\end{proposition}

\begin{proof} For example, let us assume that $j=1$ and $l=2$.
By definition
\begin{align*}
\Lambda_{\overrightarrow{s_0} ^{1,2}} (f_1,f_2,f_3) & := \sum_{n\geq 0} \ \sum_{s\in \overrightarrow{s_0} ^{1,2} \cap \OQ_n} |I_s|^{-1/2} \left|\langle f_1,\Phi_{s_1}\rangle
\langle f_2,\Phi_{s_2}\rangle \langle f_{3,n},\Phi_{s_3}\rangle\right| \\
& = |I_{s_0}|^{-1/2} \sum_{s\in \overrightarrow{s_0} ^{1}}   \left|\langle f_2,\Phi_{s_2}\rangle\right|  \sum_{n\geq 0}\  \sum_{\genfrac{}{}{0pt}{}{s'\in \OQ_n}{s'_2=s_2}} \left|\langle f_1,\Phi_{s'_1}\rangle \langle f_{3,n},\Phi_{s'_3}\rangle\right|.
\end{align*}
By Cauchy-Schwartz inequality, it follows that
\begin{align*}
\Lambda_{\overrightarrow{s_0}^{1,2}} (f_1,f_2,f_3) & \\
 &\hspace{-2cm} \leq  |I_{s_0}|^{-1/2} \sum_{s\in \overrightarrow{s_0} ^{1}}   \left|\langle f_2,\Phi_{s_2}\rangle\right|
 \left( \sum_{\genfrac{}{}{0pt}{}{s'\in \OQ_n}{s'_2=s_2}} \left|\langle f_1,\Phi_{s'_1}\rangle\right|^2 \right)^{1/2}  \left( \sum_{n\geq 0} \sum_{\genfrac{}{}{0pt}{}{s'\in \OQ_n}{s'_2=s_2}} \left| \langle f_{3,n},\Phi_{s'_3}\rangle\right|^2 \right)^{1/2} \\
 & \hspace{-2cm} \leq  \left(\sum_{s\in \overrightarrow{s_0} ^{1} } \left|\langle f_2,\Phi_{s_2}\rangle\right|^2\right)^{1/2}  \sup_{s\in \overrightarrow{s_0} ^{1} } |I_s|^{-1/2}\left( \sum_{\genfrac{}{}{0pt}{}{s'\in \OP}{s'_2=s_2}} \left|\langle f_1,\Phi_{s'_1}\rangle\right|^2 \right)^{1/2}
\left( \sum_{n\geq 0} \sum_{s'\in \overrightarrow{s_0}^{1,2} \cap \OQ_n} \left| \langle f_{3,n},\Phi_{s'_3}\rangle\right|^2 \right)^{1/2}.
 \end{align*}
Using the definition of size, we get
$$ \left(\sum_{s\in \overrightarrow{s_0} ^{1}} \left|\langle f_2,\Phi_{s_2}\rangle\right|^2\right)^{1/2} \leq |I_{s_0}|^{1/2} \overrightarrow{\textrm{size}}^1_2(f_2,\OP)$$
and
$$ \left( \sum_{n\geq 0} \ \sum_{s'\in \overrightarrow{s_0}^{1,2} \cap \OQ_n} \left| \langle f_{3,n},\Phi_{s'_3}\rangle\right|\right)^{1/2} \leq |I_{s_0}|^{1/2} \overrightarrow{\textrm{size}} ^{1,2} _3(f_3,\OP).$$
Moreover, we have
$$ \sup_{s\in \overrightarrow{s_0} ^{1} } |I_{s}|^{-1/2}\left( \sum_{\genfrac{}{}{0pt}{}{s'\in \OP}{s'_2=s_2}} \left|\langle f_1,\Phi_{s'_1}\rangle\right|^2 \right)^{1/2} \leq \overrightarrow{\textrm{size}} ^2_1(f_1,\OP).$$
This completes the proof.
\end{proof}

\mb We now prove a combinatorial algorithm which will allow us to organize the whole collection of tri-tiles
into sub collections in such a way that we have control over the different quantities associated with these sub collections so that we can sum
them up to get the desired result.
\begin{proposition} \label{prop:algo} Let $j\in \{1,2\}$ be fixed and $l\neq j$. Let $\OP$ be a collection of tri-tiles and $f$ a function such that for some integer $d\in\Z$
\be{prop1bis} \overrightarrow{\textrm{size}}^l _j(f,\OP) \leq 2^{-d}  \overrightarrow{\textrm{energy}}^l _j(f,\OP). \ee
Then we can decompose $\OP=\OP^1 \bigcup \OP^2$ so that the collection $\OP^1$ satisfies
 \be{prop1} \overrightarrow{\textrm{size}} ^l _j(f,\OP^1) \leq 2^{-d-1}  \overrightarrow{\textrm{energy}} ^l _j(f,\OP) \ee
and the collection $\OP^2=(\overrightarrow{s_i}^{l,j})_i$ consisting of $l$-vectorized tri-tiles $\overrightarrow{s_i}^{l,j}$, is $j-$disjoint and satisfies
\be{prop2} \sum_{i} |I_{s_i}| \lesssim 2^{2d}. \ee
\end{proposition}

\begin{proof} We follow ideas of Proposition 12.2 in \cite{MTT3}. Let us denote the energy $E:=\overrightarrow{\textrm{energy}}^l _j(f,\OP)$ and consider the case $l=2$ and $j=1$ (the other cases can be similarly treated).
We initialize with a collection $D$  to be the empty collection. \\
We consider the set of all tri-tiles $s \subset \OP$ satisfying
\be{hyp1} \sum_{s'\in \overrightarrow{s}^l } \left|\langle f,\Phi_{s'_j}\rangle \right|^2 \geq 4^{-1} \left(2^{-d} E \right)^{2} |I_s| . \ee
If there are no tri-tiles obeying the previous condition, we terminate the algorithm. Otherwise, we select $s_1$ among all such tri-tiles. We add tri-tiles $\overrightarrow{s_1}^{l,j}$ into the collection $D$ and remove them from the collection $\OP$.
Then we repeat the algorithm until there is no tri-tile left satisfying the selection criteria. This terminates in finitely many steps as the whole collection $\OP$ is finite. In the process, we have constructed a collection $(s_i)_{i}$ of tri-tiles. We set $\OP^2:= \cup_{i} \overrightarrow{s_i}^{l,j}$ and $\OP^1:=\OP\setminus \OP^2$. \\
We claim that the selected tri-tiles $(s_i)_i$ satisfy : the collection of $l$-vectorized tri-tiles $\overrightarrow{s_i}^l$ is
strongly $j$-disjoint. It is clear from the construction that for $i< i'$, tri-tiles $s\in \overrightarrow{s_i}^l$ and $s'\in \overrightarrow{s_{i'}}^l$ are different. In fact  we know even more that $s_j\neq s'_j$ (else $s'$ would have been removed with $\overrightarrow{s_{i}}^{l,j}$ at the $i$th step). By sparseness we deduce that $s_j\cap s'_j = \emptyset$.
 Now suppose, on the contrary, that we had tri-tiles $s\in \overrightarrow{s_{i}}^l$ and $s'\in \overrightarrow{s_{i'}}^l$ for $i\neq i'$ such that $2\omega_{s,j} \cap 2\omega_{s',j} \neq \emptyset$ and $I_{s'}\subset I_{s_i}=I_s$. By symmetry, we can assume that $i< i'$. Since $|\omega_{s,j}| \simeq |\omega_{s',j}|$, the sparseness assumption would imply that $\omega_{s',j}=\omega_{s,j}$, which is not possible since $s'$ would also be removed with $\overrightarrow{s_{i}}^{l,j}$
and cannot be found by the algorithm at the step $i'>i$. We thus arrive at a contradiction, which proves the strong $j$-disjointness of the vectorized tri-tiles.\\
It remains us to check property (\ref{prop2}) (since (\ref{prop1}) is obvious by construction), which corresponds to
$$ \sum_{i} |I_{s_i}| \lesssim 2^{2d}.$$
Since the $l$-vectorized tri-tiles $\overrightarrow{s_{i}}^{l}$ are proved to be $j$-disjoint, by using the definition of ``energy'' (see Definition \ref{def:energy1}) with (\ref{hyp1}) and (\ref{prop1bis}), it follows that
$$ 2^{-2d} E^2 \sum_{i} |I_{s_i}| \lesssim E^2,$$
which proves the desired inequality.
\end{proof}

\mb Now, it remains to prove a similar algorithm for the quantities associated with the sequence of functions. Indeed, this part is far more easy, and is proved in the following proposition.

\begin{proposition} \label{prop:algo1} Let $h:=(h_n)_n$ be a sequence of functions, $\{j,l\}=\{1,2\}$ and $\OP$ be a finite collection of tri-tiles such that for some integer $d\in\Z$
\be{prop1ter} \overrightarrow{\textrm{size}}^{j,l} _3(h,\OP) \leq 2^{-d}  \overrightarrow{\textrm{energy}}_3(h,\OP). \ee
Then we can decompose $\OP=\OP^1 \bigcup \OP^2$ where the collection $\OP^1$ satisfies
 \be{prop1111} \overrightarrow{\textrm{size}} ^{j,l} _3(h,\OP^1) \leq 2^{-d-1}  \overrightarrow{\textrm{energy}} _3(h,\OP) \ee
and $\OP^2$ is a collection of vectorized tri-tiles $\OP^2=(\overrightarrow{s_i}^{j,l})_i$ such that
\be{prop2ter} \sum_{i} |I_{s_i}| \lesssim 2^{2d}. \ee
\end{proposition}

\begin{proof} We follow the previous ideas. Let us denote the energy $E:=\overrightarrow{\textrm{energy}}_3(h,\OP)$. We initialize $D$ a collection of tri-tiles to be the empty collection. \\
We consider the set of all tri-tiles $s \subset \OP$ satisfying
\be{hyp1ter} \sum_n \sum_{s' \in \overrightarrow{s}^{j,l} \cap \OP_n}  \left|\langle h_n,\Phi_{s'_3}\rangle \right|^2 \geq 4^{-1} \left(2^{-d} E \right)^{2} |I_s |. \ee
If there is no such tri-tiles, we terminate the algorithm. Otherwise, we choose a tri-tile $s_1$ among all tri-tiles satisfying~(\ref{hyp1ter}). We add $\overrightarrow{s_1}^{j,l}$ to the collection $D$ and we remove it from the collection $\OP$.
Then we repeat the algorithm and we construct a sequences of tri-tiles $(s_i)_i$. When this algorithm is finished, we have constructed a collection $\OP^2:= \cup_{i} \overrightarrow{s_i}^{j,l}$ and $\OP^1:=\OP\setminus \OP^2$. \\
It remains us to check property (\ref{prop2ter}), which as before, is a direct consequence of the fact that $D$ is a collection of different tri-tiles by construction.
\end{proof}

\mb We can now complete the proof of our main result.

\mb {\bf Proof of Theorem \ref{thm:principal3}:} \\
The arguments are now routine. We refer the reader to \cite{MTT2, MTT3} for precise arguments. We will just give a sketch of the reasoning. Fix functions $f_1,f_2$ and a sequence of functions $f_3=(f_{3,n})_n$, then we want to estimate $\Lambda_{\OQ}(f_1,f_2,f_3)$. \\
Thanks to Propositions \ref{prop:tree}, \ref{prop:algo} and \ref{prop:algo1}, we can iterate the different combinatorial algorithms in order to get a
partition of the initial collection $\OQ$
$$ \OQ = \bigcup_{d \in \Z} \OQ^d,$$
where for each $d \in \Z$ and for all $\{j,l\}=\{1,2\}$ we have
$$ \overrightarrow{\textrm{size}}^l_j(f_j,\OQ^d) \leq 2^{-d} \overrightarrow{\textrm{energy}}^ l_j(f_j,\OQ),$$
and
$$ \overrightarrow{\textrm{size}}^{j,l}_3(f_3,\OQ^d) \leq 2^{-d} \overrightarrow{\textrm{energy}} _3(f_3,\OQ).$$
In addition, the collection $\OQ^d$ can be covered by a collection of vectorized tri-tiles for some collection $D_d$ such that
\be{arbre} \sum_{i\in D_d} |I_{s_i}| \lesssim 2^{2d}. \ee
Proposition \ref{prop:tree} yields
\begin{align*}
\Lambda_{\OQ^d}(f_1,f_2,f_3) & \lesssim \sum_{i\in D_d} |I_{s_i}| \prod_{\{j,l\}=\{1,2\}} \min\left\{2^{-d} \overrightarrow{\textrm{energy}}^ l_j(f_j,\OQ), \overrightarrow{\textrm{size}}^ l_j(f_j,\OQ) \right\}  \\
 & \hspace{2cm}\min\left\{2^{-d} \{\overrightarrow{\textrm{energy}}_3(f_3,\OQ)\}, \overrightarrow{\textrm{size}}^{1,2} _3(f_3,\OQ), \overrightarrow{\textrm{size}}^{2,1} _3(f_3,\OQ) \right\} \\
 & \lesssim 2^{2d} \prod_{\{j,l\}=\{1,2\}} \min\left\{2^{-d} \overrightarrow{\textrm{energy}}^ l_j(f_j,\OQ), \overrightarrow{\textrm{size}}^ l_j(f_j,\OQ) \right\}  \\
 & \hspace{2cm}\min\left\{2^{-d} \{\overrightarrow{\textrm{energy}}_3(f_3,\OQ)\}, \overrightarrow{\textrm{size}}^{1,2} _3(f_3,\OQ), \overrightarrow{\textrm{size}}^{2,1} _3(f_3,\OQ) \right\}.
\end{align*}
Then we can compute the sum over $d$ and get the desired inequality~:
\begin{align*}
\Lambda_{\OQ}(f_1,f_2,f_3) & \leq \sum_{d\in\Z} \Lambda_{\OQ^d}(f_1,f_2,f_3) \\
 & \lesssim \prod_{\{j,l\}=\{1,2\}} \overrightarrow{\textrm{energy}}^ l_j(f_j,\OQ) ^{1-\theta_j} \overrightarrow{\textrm{size}}^ l_j(f_j,\OQ)^{\theta_j} \\
 & \hspace{2cm}  \overrightarrow{\textrm{energy}}_3(f_3,\OQ)^{1-\theta_3} (\overrightarrow{\textrm{size}}^ {1,2} _3(f_3,\OQ)+\overrightarrow{\textrm{size}}^{2,1} _3(f_3,\OQ) )^{\theta_3},
\end{align*}
for every exponents $\theta_i\in(0,1)$ with $\theta_1+\theta_2+\theta_3=1$ (see for example Proposition 4.3 in \cite{MTT3b}.
Now we invoke Theorems \ref{thm:size1}, \ref{thm:energy1}, \ref{thm:size2} and \ref{thm:energy2} with the choice $\theta_i=1-\frac{2}{p_i}$ to
conclude that for all $f_i\in F(E_i)$ we have
\begin{align*}
\Lambda_{\OQ}(f_1,f_2,f_3) & \lesssim |E_1|^{1/p_1} |E_2|^{1/p_2} |E_3|^{1/p_3},
\end{align*}
which is the desired weak-type $(p_1,p_2,p_3)$ estimate for $\Lambda_\OQ$ and it is independent of the collection $\OQ$. Hence we conclude the proof of Theorem \ref{thm:principal3}.\findem

\part{Applications}

\section{Applications to bilinear pseudo-differential operators } \label{sec:bipseudo}

Let us first recall classes of bilinear pseudo-differential symbols. Two main types of $x$-dependent  classes of symbols have been studied in the literature.  One is the  Coifman-Meyer type class $BS_{\rho, \delta}^m(\R)$, $0\leq \delta \leq \rho \leq 1$, $m \in \R$, of symbols satisfying estimates of the form
\begin{equation}
\label{pseu}
|\partial _x^a\partial _\xi ^b\partial _\eta ^c \sigma (x, \xi ,\eta )|\leq C_{a,b,c} (1+|\xi |+|\eta |)
^{m+\delta a-\rho (b+c)},
\end{equation}
for all indices $a,b,c$. \\
The other type corresponds to classes which are denoted by $BS_{\rho, \delta; \, \theta}^m(\R)$, $0\leq \delta \leq \rho \leq 1$, $m \in \R$, $-\pi/2<\theta \leq \pi/2$, and consist of symbols satisfying
\begin{equation}
\label{tildepseu}
|\partial _x^a \partial _\xi ^b \partial _\eta ^c \sigma (x, \xi ,\eta )|\leq
C_{a,b,c ; \theta}(1+|\eta -  \tan (\theta) \xi |)^{m+\delta a-\rho (b+c)}
\end{equation}
(where for $\theta=\pi/2$ the estimates are interpreted to decay in terms of $1+|\xi|$ only).
Both the classes can be seen as bilinear analogs of the classical H\"ormander classes  $S_{\rho, \delta}^m(\R)$ .

\gb
 It is now well-understood that the operators in $BS^0_{1,0}$  are examples of certain singular integral operators and fit within the general multilinear Calder\'on-Zygmund theory developed by Grafakos and Torres \cite{gt1};  see also the work of Kenig and Stein \cite{KS}. So their boundedness properties in Lebesgue spaces are well-known. \\
 The general classes $BS_{\rho, \delta; \, \theta}^m$ with $x$-dependent symbols were first introduced in \cite{bipseudo}. For $m\geq 0$ and $(\rho,\delta)=(1,0)$, their boundedness properties were obtained by Bernicot \cite{pseudo1, pseudo}. Also some results were obtained by Bernicot and Torres \cite{BT} for the exotic classes $BS^m_{1,1;\theta}$.
These class of bilinear pseudodifferential symbols allow to build functional calculus (see \cite{pseudo,BMNT}).

 \gb As explained earlier in the introduction, this section is devoted to the study of the bilinear operators associated with symbols belonging to the exotic class $BS_{0,0}^0=BS^0_{0,0,\theta}$. As pointed out, this class does not depend on the angle $\theta$.

\begin{remark}
First we point out that the behavior on the product of modulation spaces (instead of Lebesgue spaces) of bilinear operators associated to a symbol belonging to $BS^0_{0,0}$ is well-understood, see \cite{BGHO}. Results on some products of Besov spaces also follow from well-known embeddings. \\
In addition, we refer the reader to \cite{BMNT} for the following result: every symbol $\sigma\in BS^0_{0,0}$ define a bilinear operator which is bounded from $L^2 \times W^{s,\infty}$ to $L^2$ for sufficiently large $s$.
\end{remark}

Considering boundedness from the product of Lebesgue spaces, we know from \cite{BT2} (Proposition 1) that extra assumptions on the symbol are necessary (due to some counterexamples). We also consider some smaller classes of bilinear symbols and show that the associated bilinear operators are bounded.\\
This approach was already treated in \cite{BT2} where the authors add the following assumption: for every integer $\alpha\geq 0$
\be{eq:condd} \sup_{x\in\R} \int_{\R} \left\| \partial_\xi^\alpha  m(x,\xi,\cdot) \right\|_{L^2(\R)} d\xi  <\infty \ee
and
\be{eq:condd2} \sup_{x\in\R} \int_{\R} \left\| \partial_\eta^\alpha  m(x,\cdot,\eta) \right\|_{L^2(\R)} d\eta <\infty. \ee
In \cite{BT2}, following a bilinear version of the Calder\'on and Vaillancourt reasoning, B\'enyi and Torres proved that a $BS^0_{0,0}$-symbol satisfying these two previous conditions gives a bilinear symbol which is bounded from $L^2 \times L^2$ to $L^1$. Moreover, using the Wigner transform, same results are obtained in \cite{BT2} for a symbol $\sigma$ verifying for all $i,j\in\{0,1\}$
$$ \sup_{x\in\R} \left\| \partial_\xi^i \partial_\eta^j \sigma(x,\cdot,\cdot)\right\|_{L^2(\R^2)}.$$
These two results hold in the setting of multi-dimensionnal variables.

\gb In this current work, we will define other assumptions and prove the boundedness in the local $L^2$-case for the assocoated bilinear operators. Mainly, we assume that our symbol is very smooth along the $x$-variable but we would like to require less regularity in the frequency plane.

\begin{definition} \label{def:bs}
Let $\theta\in \mathbb{S}^1$ be a unit vector, we define $\theta^\perp$  to be the orthogonal vector in $\R^2$. We say that a symbol $m:\R^3 \rightarrow \R$ belongs to the class $W^{1,s}_\theta(BS^0_{0,0})$ if $m\in C^\infty(\R^3)$ satisfies~:
\begin{itemize}
 \item the symbol $m\in BS_{0,0}$, for all index $a,b,c\in\N$
$$ \left\| \partial_x^a \partial_\xi^b \partial_\eta^c m(x,\xi,\eta) \right\|_{L^\infty(\R^3)} <\infty$$
\item the symbol $m$ belongs to $W^{1,s}$ in the direction $\theta$~: for all index $a\in\N$
$$ \sup_{x\in\R} \ \left\| \sup_{t\in\R} \left|\partial_x^a m(x,.\theta + t\theta^\perp)\right| + \left|\partial_x^a
\langle \nabla_{(\xi,\eta)} m(x,.\theta + t\theta^\perp),\theta\rangle \right| \right\|_{L^s(\R)} <\infty.$$
\end{itemize}
In this case, we set $\|m\|_{W^{1,s}_\theta(BS^0_{0,0})}$ to be the maximum of all the above constants.
\end{definition}

\mb So a symbol belonging to $W^{1,s}_\theta(BS^0_{0,0})$ is a $BS^0_{0,0}$-symbol which satisfies an extra Sobolev regularity only along the direction $\theta$ in the frequency plane.
In \cite{BT2} (see (\ref{eq:condd}) and (\ref{eq:condd2})), the assumptions were weaker in space (since we require less regularity in the variable $x$) and stronger in frequency (regularity in the two frequency variables is required).

\begin{theorem} Let $\sigma:\R^3 \rightarrow \R$ be a $BS_{0,0}$ symbol such that there exist a direction $\theta\in {\mathbb S}^1$ with $\sqrt{2}\theta\notin\{(1,-1), (-1,1)\}$ and $s\in(1,2]$ such that $\sigma\in W^{1,s}_\theta(BS^0_{0,0})$. Then the operator $T_{\sigma}$ is bounded in the local $L^2$-case, i.e. it is bounded from $L^p(\R) \times L^q(\R)$ into $L^r(\R)$ for every exponents $p,q,r' \in(2,\infty)$ such that
$$ \frac{1}{r}= \frac{1}{p}+\frac{1}{q}.$$
\end{theorem}

\begin{remark} The end-point estimate $r=2$ remains true for $x$-independent symbols $\sigma$.
\end{remark}

\begin{proof}

\mb
{\bf The case of an $x$-independent symbol:} \\
For simplicity, we will assume that $\|\sigma\|_{W^{1,s}_\theta(BS^0_{0,0})}=1$. We also assume that $\sigma$ depends only on the variable $\lambda:=\langle(\xi,\eta),\theta\rangle$, i.e. $\sigma(\xi,\eta)=m(\lambda)$ for some function $m$. The idea of the proof is the following:  we will decompose the symbol $m$ with elementary functions in order be bring the previous square functions into the picture. \\
Consider a smooth function $\chi$ supported on $[-1,1]$ such that for all $\xi\in\R$
$$ 1 = \sum_{p\in\Z} \chi(\xi-p).$$
We decompose the symbol as follows:
$$ m(\lambda) = \sum_{p\in\Z} \chi(\lambda - p) m(\lambda) := \sum_{p\in\Z} m_{p}(\lambda).$$
Note that the symbols $m_{p}$ are well-adapted to the interval $[p-1, p+1]$ with a norm bounded by
$$ \sup_{p-1\leq t \leq p+1} |m(t)| \leq \int_{p-1}^{p+1} (|m(\lambda)| + |m'(\lambda)|) d\lambda.$$
For each integer $k$, we denote $D_k$ the following set~:
$$ D_k:= \left\{p,\ 2^k\leq \int_{p-1}^{p+1} (|m(\lambda)| + |m'(\lambda)|) d\lambda <2^{k+1} \right\}.$$
Since $|m|+|m'| \in L^s(\R)$, we deduce that
\begin{align}
 1:= \|m\|_{W^{1,s}}^s \geq \||m|+|m'|\|_{L^s}^s & = \int (|m(\lambda)| + |m'(\lambda)|)^s d\eta \nonumber \\
 & \gtrsim \sum_{p}  \int_{p-1}^{p+1} (|m(\lambda)| + |m'(\lambda)|)^s d\lambda \nonumber \\
 & \gtrsim  \sum_{p}  \left(\int_{p-1}^{p+1} |m(\lambda)| + |m'(\lambda)| d\lambda\right)^s  \nonumber \\
 & \gtrsim \sum_{k} 2^{ks} (\sharp D_k). \label{eq:impp}
\end{align}
So only for $k$ satisfying  $2^{k}\lesssim 1$, the contribution to the above sum is non null.
We now estimate
\begin{align*}
\left\|T_{\sigma}(f,g)\right\|_{L^r(\R)} & \leq \left\|\sum_{p\in\Z} |T_{m_{p}}(f,g)| \right\|_{L^r(\R)} \\
 & \leq \sum_{\genfrac{}{}{0pt}{}{k\in\Z}{2^k \lesssim 1 }}  \left\| \sum_{p\in D_k} |T_{m_{p}}(f,g)|\right\|_{L^r(\R)} \\
& \leq \sum_{\genfrac{}{}{0pt}{}{k\in\Z}{2^k \lesssim 1 }} (\sharp D_k)^{1/2} \left\| \left(\sum_{p\in D_k} \left|T_{m_{p}}(f,g)\right|^2\right)^{1/2} \right\|_{L^r(\R)},
\end{align*}
where we used Cauchy-Schwartz inequality.
Then, from what we described previously and Theorem \ref{thm:principal}, we know that
$$\left\| \left(\sum_{p\in D_k} \left|T_{m_{n,p}}(f,g)\right|^2\right)^{1/2} \right\|_{L^r(\R)}\lesssim \|f\|_{L^p(\R)} \|g\|_{L^q(\R)} 2^{k}.$$
As a consequence, we conclude that
\begin{align*}
 \left\|T_{\sigma}(f,g)\right\|_{L^r(\R)} & \lesssim  \|f\|_{L^p(\R)} \|g\|_{L^q(\R)} \sum_{\genfrac{}{}{0pt}{}{k\in\Z}{2^k \lesssim 1 }} (\sharp D_k)^{1/2} 2^{k}.
\end{align*}
Finally using (\ref{eq:impp}), we obtain
\begin{align*}
 \left\|T_{\sigma}(f,g)\right\|_{L^r(\R)} & \lesssim  \|f\|_{L^p(\R)} \|g\|_{L^q(\R)} \left(\sum_{k} (\sharp D_k) 2^{ks}\right)^{1/2} \left(\sum_{\genfrac{}{}{0pt}{}{k\in\Z}{2^k \lesssim 1}}  2^{2k-ks} \right)^{1/2} \\
& \lesssim   \|f\|_{L^p(\R)} \|g\|_{L^q(\R)} ,
\end{align*}
where we have used that $s<2$.

\gb Thus, we have proved the desired result for those $x$-independent symbols which depend only on the above variable $\lambda:=\langle(\xi,\eta),\theta\rangle$. We leave to the reader to verify the general case when the symbol $\sigma$ depends on the two variable $\xi,\eta$. We produce the same reasoning: by decomposing $(\xi,\eta)$ into the basis $(\theta,\theta^{\perp})$ and then work on the  main variable $\langle(\xi,\eta),\theta\rangle$. All the estimates remain true since the other variable  $\langle(\xi,\eta),\theta^\perp\rangle$ does not play a role and the assumptions are uniform on it.

\gb
{\bf The case of an $x$-dependent symbol:} We refer the reader to \cite{pseudo1} for details concerning the ``general'' principle which allow us to obtain boundedness for $x$-dependent symbols by knowing the boundedness for corresponding $x$-independent symbols. This consists of two steps, we briefly recall it.\\
For $x$-independent symbol, the previous decomposition only uses square functions associated with intervals of lengths equivalent to one. Using the ideas of \cite{pseudo1}, we can obtain ``off-diagonal decay'' at the scale $1$ for the bilinear operators associated with $x$-independent symbols. More precisely, we can prove that for every interval $I$ of length $1$ and $m$ a bilinear multiplier belonging to $W^{1,s}_\theta(BS^0_{0,0})$, there exists $\delta>0$ (as large as we want) such that
\begin{align}
\left( \frac{1}{|I|} \int_{I} \left|T_m(f,g)(x)\right|^r dx \right)^{1/r}  & \nonumber \\
  &\hspace{-2cm} \lesssim \left[\sum_{k\geq 2} 2^{-k\delta} \left( \frac{1}{|2^{k+1}I|} \int_{2^kI\setminus 2^{k-1}I} |f(x)|^p dx \right)^{1/p} + \left(\int_{2I} |f(x)|^p dx \right)^{1/p} \right] \nonumber \\
 &  \hspace{-2cm} \ \left[\sum_{k\geq 2} 2^{-k\delta} \left( \frac{1}{|2^{k+1}I|} \int_{2^kI\setminus 2^{k-1}I} |g(x)|^q dx \right)^{1/q} + \left(\int_{2I} |g(x)|^q dx \right)^{1/q} \right]. \label{eq:off}
 \end{align}
This improvement is obtained in exactly the same way as described in the previous sections with the following slightly sharper estimate. We first decompose the tiles $s$ according to the distance between the space intervals $d(I_s,I)$. Then, it is important to get fast decay according to this quantity. To achieve this, we require new bounds for the "size" quantities, so we need to prove (instead of Theorem \ref{thm:size1}) that for $j=1,2$ and $\OP$ a collection of tri-tiles, we have
$$ \overrightarrow{\textrm{size}}^l_j(f,\OP) \lesssim \sup_{s\in\OQ} \left(\frac{1}{|I_s|} \int_\R \left(1+\frac{d(x,I_s)}{|I_s|}\right)^{-N} |f(x)|^2 dx\right)^{1/2},$$
for a large enough integer $N$. We again leave it to the reader to check that the proof described for Theorem \ref{thm:size1} allows us to get this stronger estimate. Then this improvement gives us the off-diagonal decay for the bilinear operator $T_m$ (as detailed in \cite{pseudo1}).\\
Then, by Sobolev inequality, we can extend these local ``off-diagonal'' estimates for $x$-dependent symbols $\sigma$ belonging to
$W^{1,s}_\theta(BS^0_{0,0})$ (see \cite{pseudo1} for details). As a consequence by summing these off-diagonal estimates (for $I=[n,n+1]$, $n\in\Z$), we conclude the boundedness of the associated bilinear operators on the whole space.
\end{proof}

\mb The main result of this section now can be obtained as a corollary to the previous theorem. We recall that the range of exponents in the local-$L^2$ case is invariant under duality.

\begin{theorem} Let $m:\R^3 \rightarrow \R$ be a $BS_{0,0}$ symbol. Suppose that there exist a direction $\theta$ and $s\in(1,2]$ such that $m\in W^{1,s}_\theta(BS^0_{0,0})$. Then the operator $T_{m}$ is bounded in the local $L^2$-case: from $L^p(\R) \times L^q(\R)$ into $L^r(\R)$ for every exponents $p,q,r' \in(2,\infty)$ such that
$$ \frac{1}{r}= \frac{1}{p}+\frac{1}{q}.$$
Moreover, the estimates are uniform with respect to $\theta\in {\mathbb S}^1$.
\end{theorem}

\begin{proof} The previous theorem gives us the result for all the directions $\theta$, except one degenerate direction. The idea is then to use duality in order to get around this technical problem. \\
Let us first remark that for an $x$-independent symbol $m$ belonging to $W^{1,s}_\theta(BS^0_{0,0})$, we have the following descriptions of the two adjoint operators~: $T_m^{*1}:=T_{m_1^*}$ and $T_m^{*2}:=T_{m_2^*}$
 with
 \begin{eqnarray*}
 m_1^*(\xi,\eta)  = \overline{m(-\xi-\eta,\eta)} \\
 m_2^*(\xi,\eta)  = \overline{m(\xi,-\eta-\xi)}.
 \end{eqnarray*}
 Hence, $m\in W^{1,s}_\theta(BS^0_{0,0})$ is equivalent to $m_1^*\in W^{1,s}_{\theta_1}(BS^0_{0,0})$ which further is equivalent to $m_2^* \in W^{1,s}_{\theta_2}(BS^0_{0,0})$ with
$$ \cot(\theta) + \cot(\theta^{*1}) = -1 \qquad \textrm{and} \qquad \tan(\theta)+\tan(\theta^{*2})=-1,$$
where $\cot$ and $\tan$ are naturally defined in ${\mathbb S}^1$. \\
So the off-diagonal decay obtained for multipliers associated with the class $W^{1,s}_{\theta}(BS^0_{0,0})$ can be transferred to the dual classes $W^{1,s}_{\theta_1}(BS^0_{0,0})$ and $W^{1,s}_{\theta_2}(BS^0_{0,0})$. Then, as described in the previous proof, we can get boundedness on the whole space for the $x$-dependent symbols belonging to the same classes. \\
Since the previous theorem gives the result for all the direction except one, this reasoning relying on duality allows us to get the result for this specific direction. Moreover, it is easy to check that the estimates are uniform with respect to the directions $\theta\in {\mathbb S}^1$.
\end{proof}

\section{Proof of Theorem \ref{ms} for remaining index $1<p_3\leq 2$}\label{remainingindex}

We provide here the reasoning for the remark we made in the introduction concerning the validity of Theorem \ref{ms} for remaining index $1<p_3\leq2$ using our main result Theorem \ref{thm:principa}.

\begin{proof} Let $\phi \in \mathcal{S}(\R)$ be as in Theorem \ref{ms}. Consider a smooth function $\chi$ supported on $[-1,1]$ such that for all $\xi\in\R$
$$ 1 = \sum_{p\in\Z} \chi(\xi-p).$$
We decompose the function $\phi$ in a similar fashion as in previous section :
$$ \phi(\xi) = \sum_{p\in\Z} \chi(\xi - p) \phi(\xi) := \sum_{p\in\Z} \phi_{p}(\xi).$$
Using Minkowski's inequality we have
\begin{eqnarray*}
\left\| \left(\sum\limits_{n\in\Z}|T_{\phi_n}(f,g)|^2\right)^{\frac{1}{2}} \right\|_{L^{p_3}(\R)}&=& \left\| \left( \sum\limits_{n\in\Z} \left|\sum\limits_{p\in \Z}T_{\phi_{p,n}}(f,g) \right|^2 \right)^{\frac{1}{2}} \right\|_{L^{p_3}(\R)}\\
&\leq&\sum_{p\in \Z} \left\| \left(\sum\limits_{n\in\Z} \left|T_{\phi_{p,n}}(f,g) \right| ^2 \right) ^{\frac{1}{2}} \right\|_{L^{p_3}(\R)},
\end{eqnarray*}
where $\phi_{p,n}(\xi)=\phi_{p}(\xi-n)$.

\mb Note that for each $p\in \Z$, the symbol $\phi_{p}$ is well-adapted to the interval $[p-1, p+1]$, hence we can apply Theorem~(\ref{thm:principa}) to the square function associated with $\phi_{p,n}$, to conclude
\begin{eqnarray*}
\left\| \left(\sum\limits_{n\in\Z} \left|T_{\phi_{p,n}}(f,g) \right| ^2 \right) ^{\frac{1}{2}} \right\|_{L^{p_3}(\R)}&\leq& C\|\phi_p\|_{L^\infty(\R)} \|f\|_{L^{p_1}(\R)} \|g\|_{L^{p_2}(\R)},
\end{eqnarray*}
where exponents $p_1, p_2,$ and $p_3$ satisfy the Local-$L^2$ condition.
For each $p\in \Z$ using the decay of $\phi$ we know that for $N\in \N$, there exists a constant $C_N$ such that
$$\|\phi_p\|_{L^\infty(\R)} \leq \frac{C_N}{(1+|p|)^N}.$$
Hence summing over $p$, we get the desired result.
\end{proof}

\end{document}